\documentclass[]{article}
\usepackage[utf8]{inputenc}

\date{}

\usepackage[utf8]{inputenc}
\usepackage[a4paper, total={17cm, 27cm}]{geometry}
\usepackage{amsmath, amsfonts, amsthm}
\usepackage{bm}
\usepackage[ruled]{algorithm2e}
\DeclareMathOperator*{\argmin}{arg\,min}
\newcommand{\pf}{\varphi}
\usepackage{todonotes}
\usepackage{subcaption}

\usepackage{tikz}
\usepackage{pgfplots,pgfplotstable}
\pgfplotsset{compat=newest}

\definecolor{blau}{RGB}{0 144 188}
\definecolor{newblue1}{RGB}{0 144 188}
\definecolor{newblue2}{RGB}{197 216 227}
\definecolor{newgreen1}{RGB}{0 144 118}
\definecolor{newgreen2}{RGB}{197 222 215}
\definecolor{neworange1}{RGB}{255 137 0}
\definecolor{neworange2}{RGB}{255 205 105}
\definecolor{newred1}{RGB}{254 54 41}
\definecolor{newred2}{RGB}{141 26  18}
\definecolor{newpurple1}{RGB}{196 19 252}
\definecolor{newpurple2}{RGB}{93 14 117}

\newtheorem{lemma}{Lemma}
\newtheorem{theorem}{Theorem}
\newtheorem{remark}{Remark}
\newtheorem{prop}{Proposition}

\newtheorem{corollary}{Corollary}

\usepackage{authblk}

\title{A robust solution strategy for the Cahn-Larch\'e equations}

\author[1]{Erlend Storvik\footnote{Corresponding author: erlend.storvik@uib.no}}
\author[1]{Jakub Wiktor Both}
\author[1]{Jan Martin Nordbotten}
\author[1]{Florin Adrian Radu}

\affil[1]{Center for Modeling of Coupled Subsurface Dynamics, Department of Mathematics, University of Bergen, All\'egaten 44, 5007 Bergen, Norway}

\date{}

\begin{document}

\maketitle

\begin{abstract}
In this paper we propose a solution strategy for the Cahn-Larch\'e equations, which is a model for linearized elasticity in a medium with two elastic phases that evolve subject to a Ginzburg-Landau type energy functional. The system can be seen as a combination of the Cahn-Hilliard regularized interface equation and linearized elasticity, and is non-linearly coupled, has a fourth order term that comes from the Cahn-Hilliard subsystem, and is non-convex and nonlinear in both the phase-field and displacement variables. We propose a novel semi-implicit discretization in time that uses a standard convex-concave splitting method of the nonlinear double-well potential, as well as special treatment to the elastic energy. We show that the resulting discrete system is equivalent to a convex minimization problem, and propose and prove the convergence of alternating minimization applied to it. Finally, we present numerical experiments that show the robustness and effectiveness of both alternating minimization and the monolithic Newton method applied to the newly proposed discrete system of equations. We compare it to a system of equations that has been discretized with a standard convex-concave splitting of the double-well potential, and implicit evaluations of the elasticity contributions and show that the newly proposed discrete system is better conditioned for linearization techniques.
\end{abstract}

\section{Introduction}
The Cahn-Larch\'e system models elastic deformation within a two-phase solid material. Here, the solid phases evolve subject to a Ginzburg-Landau type energy functional, as proposed in the work of Cahn and Hilliard \cite{CH1, CH2}, additively coupled with the elastic energy of the system. The equations are credited to the work of Cahn and Larch\'e \cite{larche1973linear, CL} which considered stress effects related to diffusion in solids. More recently, the equations were studied experimentally and verified in \cite{CLExperiment} as a model for the connection between chemical and mechanical processes in alloys. Additionally the Cahn-Larch\'e system has been applied in relation to tumor modelling \cite{garcketumormechanics, garcke2021sparse, fritz2020subdiffusive}, diffusional corsening in solders \cite{dreyer2001modeling,graser2014numerical}, and to model the process of intercalation of lithium ions into silicon \cite{meca2018sharp}. Moreover, in \cite{luiz2009phase} a phase-field model, closely related to the Cahn-Hilliard equation was proposed to account for unsaturated flow through porous materials. Extensions to a Cahn-Larch\'e setting could be considered to model flow through swelling deformable porous media. 

Over the last two decades there has been extensive research on the well-posedness and analysis of both the continuous and discrete counterparts of Cahn-Larch\'e systems. In \cite{bonetti2002model, garcke2003CHE} existence and uniqueness results are obtained for the weak system of equations, in  \cite{garcketumormechanics} similar results are obtained for the coupling of Cahn-Larch\'e to transport in the context of tumor growth, and on the same model an optimal control problem is analyzed in \cite{garcke2021sparse}. In \cite{graser2014numerical}, existence and uniqueness of a discretized Cahn-Larch\'e system is provided, and in \cite{abels2015sharpCHE, garcke2006asymptotic} the sharp interface limit of the equations is showed to be equivalent to a modified Hele-Shaw system coupled with elasticity. There are several published works on numerical discretization techniques for the system. In \cite{graser2014numerical,feng2009fourier}, adaptive mesh refinement techniques are discussed and \cite{garcke2005CHENumerics,garcke2001cahn} consider spatial discretization with linear finite elements together with the implicit Euler and Crank-Nicholson time discretizations. 

In this work, we propose a novel semi-implicit time-discretization that corresponds to the optimality conditions of a convex minimization problem, and therefore is suitable for nonlinear solvers. The semi-implicit time discretization is related to the unconditionally gradient stable convex-concave splitting method that Eyre proposed in \cite{eyre1998} for the double-well potential of the Cahn-Hilliard equation. Here, that treatment is adopted and applied to the Cahn-Larch\'e equations, in two different settings; when the elasticity tensor is independent of, and dependent on the phase-field. In the former case, the coupling between phase-field and elasticity is linear and by evaluating the terms from the elasticity subsystem implicitly the discrete system of equations is identified with a convex minimization problem, similar to the treatment in \cite{garcke2005CHENumerics}. Furthermore, the system of equations is showed to be unconditionally gradient stable, and that an alternating minimization technique, alternating between solving for phase-field and displacement, applied to the proposed minimization problem converges. In the second case, however, implicit evaluation in time of the terms corresponding to the elasticity subsystem does not lead to a convex minimization problem when the elasticity tensor depends on the phase-field, even with the convex-concave splitting method applied to the double-well potential \cite{eyre1998}. We show through numerical examples that the Newton method fails to converge in several instances in this case and propose a way to carefully evaluate some terms explicitly in time, such that the corresponding minimization problem is convex. This leads to a system that is better conditioned for solution algorithms, and a theoretical proof of convergence for the  alternating minimization method is provided. Moreover, convergence is experienced for the Newton method in all numerical examples.

When solving the coupled discrete system of equations there exists  two common choices: Either, to solve the entire system monolithically, using some linearization procedure, or to apply an iterative decoupling method. A beneficial trait of decoupling methods is the possibility to use readily available solvers for each subsystem. For the discrete system of equations that we present in this paper that corresponds to solving an extended Cahn-Hilliard equation with well-behaving nonlinearities, due to the convex-concave splitting method, and an elasticity equation with heterogeneous elasticity tensor subsequently. For the Cahn-Hilliard subsystem some linearization technique (e.g., Newton's method) is still needed to handle the nonlinearities corresponding to the modified double-well potential and terms that arise from the elasticity contribution. The elasticity subsystem, on the other hand, reduces to a standard elasticity equation with, possibly, heterogeneous elasticity tensor. Any readily available solvers and preconditioners for these subproblems can be applied, and combining the decoupling method with the linearization of the nonlinear Cahn-Hilliard subsystem (doing only one linearization iteration in each decoupling iteration) as discussed in \cite{illiano2021iterative,both2019anderson} is possible as well. Decoupling techniques are often also known as staggered solution strategies, splitting schemes or alternating minimization for symmetric problems with an underlying minimization structure, and have been widely adopted to solve equations related to phase-field modelling of brittle fracture propagation  \cite{gerasimov2016line, farrell2017linear, storvik2021accelerated, brun2020iterative, Wick2020}, and poroelasticity equations where flow and elasticity is coupled \cite{jakubaml, storvik2019optimization, mikelicwheeler, jakubgradientflow}. Moreover, a staggered solution strategy was used to solve finite-strain elasticity coupled with the Cahn-Hilliard equation in \cite{areias}.

Here, we investigate the properties of both monolithical solvers and decoupling methods for the Cahn-Larch\'e equations. Moreover, we properly address the theoretical convergence properties of alternating minimization. To do this we formulate the discretized system of equations as a minimization problem and utilize an abstract convergence result for alternating minimization provided in  \cite{jakubAM}. This framework requires at least convexity of the minimization problem in each variable, and Lipschitz continuity of its gradients. We prove that this holds true for the discretized Cahn-Larch\'e equations and obtain convergence rates that we investigate through numerical examples. Moreover, it can be useful to apply the Anderson acceleration \cite{anderson1965iterative} post-processing technique (as done in e.g., \cite{storvik2021accelerated, both2019anderson}) to enhance the convergence speed of the alternating minimization method. This is particularly useful for staggered solution methods as the Anderson acceleration is known to be accelerating for linearly convergent fixed-point schemes \cite{evans2020proof}.

To summarize, the main contributions of the paper are:
\begin{itemize}
    \item We propose a new, semi-implicit time discretization of the Cahn-Larch\'e equations that leads to a nonlinear system which is suitable for linearization and decoupling methods.
    \item Identification of the proposed discretized equations with a convex minimization problem.
    \item A proof of convergence for alternating minimization as an iterative solver, including convergence rates.
    \item Numerical experiments showing the efficiency of the proposed time-discretization and iterative solver with comparison to monolithic methods and acceleration.
\end{itemize}
Moreover, we stress that the time-discretization and decoupling procedures that we apply here, can be extended and applied to similar models, e.g., the Cahn-Hilliard-Biot model \cite{storvikCHB}, tumor growth models with transport effects \cite{garcketumormechanics}, phase-field models for precipitation and dissolution processes \cite{bringedal2020phase} and the two-phase two fluxes Cahn-Hilliard model \cite{cances2021finite}.

The paper is structured as follows: The mathematical model and assumptions on the model parameters are presented in Section~\ref{sec:model}. In Section~\ref{sec:discretization}, we discuss the discrete problem associated with the Cahn-Larch\'e system both for constant and phase-field-dependent elasticity tensor. Moreover, we show equivalence between the discrete model and a minimization problem, and prove convergence of alternating minimization applied to this problem. In Section~\ref{sec:numerics}, we present several numerical experiments and show the benefits of the proposed discretization and linearization/decoupling method compared to standard choices. Finally, in Section~\ref{sec:conclusion} we make concluding remarks.

\section{The mathematical problem and assumptions on model parameters}\label{sec:model}
The Cahn-Larch\'e system is a combination of a Cahn-Hilliard phase-field model and linearized elasticity with infinitesimal strains and displacements \cite{larche1973linear, garcke2005CHENumerics}. We consider the domain $\Omega\subset \mathbb{R}^d$ with Lipschitz boundary, where $d$ is the spatial dimension, and the time interval $[0,T]$ with final time $T$. Let $\varphi: \Omega\times[0,T]\rightarrow [-1,1]$ be the phase-field variable, where pure phases are attained for $\varphi = -1$, $\varphi = 1$. Moreover, let $\bm u: \Omega\times[0,T] \rightarrow \mathbb{R}^d$ be the infinitesimal displacement.
\subsection{Balance laws and constitutive relations}
 We assume that the phase-field $\pf$ follows the balance law
 \begin{equation*}
     \partial_t \varphi + \nabla \cdot \bm J = R,
 \end{equation*}
 where $\bm J$ is the phase-field flux and $R$ accounts for reactions. Moreover, the stress follows quasi-static linear momentum balance (ignoring inertial effects)
\begin{equation*}
    -\nabla\cdot \bm \sigma = \bm f,
\end{equation*}
where $\bm\sigma$ is the stress-tensor and $\bm f$ corresponds to external forces. The free energy $\mathcal{E}(\varphi, \bm u)$ of the system is assumed to be an additive combination of the regularized interface energy $\mathcal{E}_\mathrm{ch}(\varphi)$ and the potential elastic energy  $\mathcal{E}_\mathrm{e}(\varphi, \bm u)$
\begin{equation}\label{eq:freeenergy}
\mathcal{E}(\pf,\bm u) := \mathcal{E}_\mathrm{ch}(\varphi) + \mathcal{E}_\mathrm{e}(\varphi, \bm u).
\end{equation}
The regularized chemical energy of the system is defined as 
\begin{equation}\label{eq:chenergy}
     \mathcal{E}_\mathrm{ch}(\pf):=  \int_\Omega \gamma\left(\frac{1}{\ell}\Psi(\pf) + \frac{\ell}{2}|\nabla\pf|^2\right)\;dx,
\end{equation}
where $\Psi(\varphi)$, often chosen as $\Psi(\varphi) = \left(1-\pf^2\right)^2$, is a double-well potential that penalizes non-pure phase-field values ($|\pf|\neq 1$), and 
$\frac{|\nabla\pf|^2}{2}$ regularizes the transition between phases by penalizing rapid changes (in space) of the phase-field. The parameter $\gamma$ is related to the interfacial tension between the two phases, and can be considered to account for adhesive/cohesive forces between the phases, and $\ell$ is related to the width of the regularization region.
The elastic potential energy is 
\begin{equation}
    \mathcal{E}_\mathrm{e}(\pf, \bm u) := \frac{1}{2} \int_\Omega \left(\bm\varepsilon(\bm u)-\xi \left(\pf-\tilde{\pf}\right)\bm I\right):\mathbb{C}(\pf)\left(\bm\varepsilon(\bm u)-\xi \left(\pf-\tilde{\pf}\right)\bm I\right)\; dx
\end{equation}
where $\bm\varepsilon (\bm u) := \frac{\nabla \bm u+\nabla \bm u^\top}{2}$ is the linearized symmetric strain tensor, $\mathbb{C}(\pf)$ is the fourth order elasticity tensor, the term $\xi \left(\pf-\tilde{\pf}\right) \bm I$ accounts for swelling effects where $\tilde{\pf}$ is a reference phase-field, and $\bm I$ is the identity tensor in $\mathbb{R}^{d\times d}$. For the rest of the paper, we assume that $\tilde{\pf} = 0$ to make the notation more simplistic. All the theory and numerical examples can trivially be extended to account for $\tilde{\pf}\in [-1,1]$.

As constitutive relations we assume that the phase-field flux $\bm J$ is diffusive and follows Fick's law
\begin{equation*}
    \bm J = -m(\pf) \nabla \mu,
\end{equation*}
where $m(\pf)$ is the chemical mobility, which we will assume to be constant in this work, and $\mu$ is the chemical potential, which is defined as the rate of change, variational derivative, of the free energy of the system with respect to the phase-field. Here, we denote the variational derivative of $\mathcal{E}$ with respect to $y$ by $\delta_y\mathcal{E}$, and standard computations yield
\begin{eqnarray*}
    \mu := \delta_\pf \mathcal{E}(\varphi, \bm u) &=& \gamma\left(\frac{1}{\ell}\Psi'(\varphi) -\ell \Delta \pf\right)- \xi \bm I:\mathbb{C}(\pf)\left(\bm\varepsilon(\bm u)-\xi \pf\bm I\right) \\&& +\frac{1}{2}\left(\bm\varepsilon(\bm u)-\xi \pf\bm I\right):\mathbb{C}'(\pf)\left(\bm\varepsilon(\bm u)-\xi \pf\bm I\right),
\end{eqnarray*}
where, we have utilized that the normal derivative of the phase-field vanishes on the boundary ($\nabla\pf \cdot\bm n = 0$ at  $\partial\Omega$). 
The stress tensor $\bm \sigma$ is defined as the rate of change of the free energy with respect to strain $\bm\varepsilon$
\begin{equation*}
    \bm \sigma := \delta_{\bm \varepsilon}\mathcal{E}(\pf, \bm \varepsilon(\bm u)) = \mathbb{C}(\pf)\left(\bm\varepsilon(\bm u) - \xi\pf\bm I\right).
\end{equation*}

In total, we search for the triplet $(\pf, \mu, \bm u)$ such that
\begin{eqnarray}
\partial_t \varphi - \nabla \cdot (m\nabla \mu) &= R \label{eq:ch1} \quad&\mathrm{in}\quad \Omega\times[0,T],\\
\mu +\gamma\left(\ell\Delta \varphi - \frac{1}{\ell}\Psi'(\varphi)\right) - \delta_\varphi\mathcal{E}_\mathrm{e}(\varphi, \bm u) &=  0\label{eq:ch2}\quad&\mathrm{in}\quad \Omega\times[0,T],\\
-\nabla \cdot\left(\mathbb{C}(\varphi)\left({\bm \varepsilon}(\bm u)-\xi\varphi\bm I\right)\right) &= \bm f\label{eq:elasticity}\quad&\mathrm{in}\quad \Omega\times[0,T],
\end{eqnarray}
with the boundary conditions $\nabla\pf\cdot \bm n = \nabla\mu\cdot\bm n = 0$ and $\bm u = \bm u_b$ on $\partial\Omega\times[0,T]$, and initial condition $\pf = \pf_0$ in $\Omega\times \{0\}$. For completeness, we mention that
\begin{equation}\label{eq:elasticenergyderivative}
    \delta_\varphi\mathcal{E}_\mathrm{e}(\varphi, \bm u) = \frac{1}{2}\left({\bm \varepsilon}(\bm u) - \xi\pf\bm I\right)\!:\!\mathbb{C}'(\varphi)\left({\bm \varepsilon}(\bm u) - \xi\pf\bm I\right) - \xi\bm I\!:\!\mathbb{C}(\pf)\left({\bm \varepsilon}(\bm u) - \xi\pf\bm I\right),
\end{equation}
where the elasticity tensor $\mathbb{C}(\pf)$ is depending on the phase-field through the interpolation function $\pi(\pf)$; $\mathbb{C}(\pf) = \mathbb{C}_{-1} + \pi(\pf)(\mathbb{C}_1 -\mathbb{C}_{-1})$, and we assume for simplicity to have homogeneous Dirichlet boundary conditions for the elasticity subproblem, i.e., $\bm u_b = 0$.

\subsection{Phase-field independent elasticity tensor}
A simplified model is obtained in the special case of phase-field independent elasticity tensor $\mathbb{C}(\pf) = \mathbb{C}$. We consider it as a special case here because it is a popular simplification to the system, and the analysis of it will make the foundation for the numerical solution strategies for the situations where the elasticity tensor depends on the phase-field. The system \eqref{eq:ch1}--\eqref{eq:elasticity} now becomes: Find $(\pf, \mu, \bm u)$ such that
\begin{eqnarray}
\partial_t \varphi - \nabla \cdot (m \nabla \mu) &= R \label{eq:homoch1}\quad&\mathrm{in}\quad \Omega\times[0,T],\\
\mu +\gamma\left(\ell\Delta \varphi - \frac{1}{\ell}\Psi'(\varphi)\right) + \xi \bm I:\mathbb{C}\left(\bm\varepsilon(\bm u) - \xi\pf\bm I\right) &= 0\label{eq:homoch2}\quad&\mathrm{in}\quad \Omega\times[0,T],\\
-\nabla \cdot\left(\mathbb{C}\left({\bm \varepsilon}(\bm u)-\xi\pf\bm I\right)\right) &= \bm f\label{eq:homoelasticity}\quad&\mathrm{in}\quad \Omega\times[0,T],
\end{eqnarray}
with the boundary conditions $\nabla\pf\cdot \bm n = \nabla\mu\cdot\bm n = 0$ and $\bm u = 0$ on $\partial\Omega\times[0,T]$, and initial condition $\pf = \pf_0$ in $\Omega\times \{0\}$.

\begin{remark}
 Notice that the equations \eqref{eq:ch1} and \eqref{eq:homoch1} imply that the total phase-field is balanced in time by the reaction term
 \begin{equation}
     \partial_t \int_\Omega \pf \; dx = \int_\Omega R \; dx
 \end{equation} due to the homogeneous Neumann boundary conditions on $\mu$.
\end{remark}

\subsection{Assumptions on material parameters}
In this paper we will use the following assumptions on the model:
\begin{itemize}
    \item[(A1)] We require that the double-well potential has a convex-concave splitting 
    \begin{equation*}
        \Psi(\pf) = \Psi_c(\pf) -\Psi_e(\pf),
    \end{equation*}
    where $\Psi_c(\pf)$ and $\Psi_e(\pf)$ are convex functions, and that the derivative of the convex part $\Psi_c'(\pf)$ is Lipschitz continuous
    \begin{equation*}
        \left(\Psi_c'(\pf_1)-\Psi_c'(\pf_2)\right)(\pf_1 - \pf_2)\leq L_{\Psi_c}(\pf_1-\pf_2)^2, \quad \forall \pf_1, \pf_2\in \mathbb{R},
    \end{equation*}
    with Lipschitz constant $L_{\Psi_c}$. 
    The convex-concave splitting of the classical double-well potential does not satisfy this assumption, since the Lipschitz constant of the convex part is not bounded. To rectify this situation, we modify the double-well potential outside the interval $(-\theta ,\theta)$, for some choice of $\theta>1$, in the following way:
     \begin{equation*}
        \Psi(\pf) = 
        \begin{cases}
            2\left(\theta^2-1\right)\pf^2-\left(\theta^4-1\right),& \quad \pf\geq \theta,\\
            (1-\pf^2)^2,& \quad \pf\in(-\theta,\theta),\\
            2\left(\theta^2-1\right)\pf^2-\left(\theta^4-1\right),& \quad \pf\leq -\theta,
        \end{cases}
    \end{equation*}
    which is split into the convex functions
    \begin{equation*}
        \Psi_c(\pf) = 
        \begin{cases}
            2\theta^2\pf^2-\left(\theta^4-1\right),& \quad \pf\geq \theta,\\
            \pf^4+1,& \quad \pf\in(-\theta,\theta),\\
            2\theta^2\pf^2-\left(\theta^4-1\right),& \quad \pf\leq -\theta,
        \end{cases}
    \end{equation*}
    and
    \begin{equation*}
        \Psi_e(\pf) = 2\pf^2.
    \end{equation*}
    This modification ensures the uniformly bounded Lipschitz continuity of $\Psi_c'$, with bound $L_{\Psi_c} = 2\theta^2$, without altering the solution to the problem, since the phase-field rarely takes values outside $[-1,1]$.
    \item[(A2)] There exist constants $c_\mathbb{C}>0$ and $C_\mathbb{C}>0$ such that 
    \begin{equation}\label{eq:posdefC}
        c_\mathbb{C}\|\bm e\|^2_{L^2(\Omega)}\leq \left(\mathbb{C}(s)\bm e;\bm e\right)\leq C_\mathbb{C}\|\bm e\|^2_{L^2(\Omega)}
    \end{equation}
    for all symmetric second order tensor functions $\bm e\in L^2(\Omega)$ and scalar functions $s\in L^\infty(\Omega)$, where $( \cdot ; \cdot )$ is the $L^2(\Omega)$ tensor inner-product. It follows that $(\bm e, \bm w) \mapsto \left(\mathbb{C}(s)\bm e;\bm w\right)$ defines an inner-product on $L^2(\Omega)$, hence we have the Cauchy-Schwarz'-type inequality 
    \begin{equation}\label{eq:tensorCS}
        \left(\mathbb{C}(s)\bm e;\bm w\right)\leq \left(\mathbb{C}(s)\bm e;\bm e\right)^\frac{1}{2}\left(\mathbb{C}(s)\bm w;\bm w\right)^{\frac{1}{2}}.
    \end{equation}
\end{itemize}

\section{Numerical solution strategies for the Cahn-Larch\'e equations}\label{sec:discretization}
We now consider numerical solution strategies for the Cahn-Larch\'e equations with the aim of establishing an efficient and robust solver. At first, in Section~\ref{sec:homo}, a solution strategy for the system with phase-field independent elasticity tensor \eqref{eq:homoch1}--\eqref{eq:homoelasticity} is proposed. Then, in Section~\ref{sec:hetero}, the equations with phase-field dependent elasticity tensor \eqref{eq:ch1}--\eqref{eq:elasticity} are considered.

\subsection{Notation, variational system of equations and discrete function spaces}
Throughout the paper $(\cdot,\cdot)$ will denote the $L^2(\Omega)$ inner product for scalar- and vector-valued functions, $\langle\cdot,\cdot\rangle$ is the duality pairing, and $\langle\cdot,\cdot\rangle_X$ represents specific inner products defined on the Hilbert space $X$.
We consider the following continuous variational formulation of the system \eqref{eq:ch1}--\eqref{eq:elasticity}: Find $\left(\pf, \mu, \bm u\right)\in H^1\left([0,T],H^1(\Omega)\right)\times L^2\left([0,T],H^1(\Omega)\right)\times L^2\left([0,T],\left(H^1_0(\Omega)\right)^d\right)$ such that
\begin{eqnarray}
\left( \partial_t \pf,q^\pf \right) + \left( m \nabla\mu,\nabla q^\pf\right) -\left(  R ,q^\pf\right)&=&0 \label{eq:varch1} \\
 \left( \mu,q^\mu\right)- \gamma\ell\left( \nabla\pf,\nabla q^\mu\right) - \frac{\gamma}{\ell}\left(  \Psi'(\pf), q^\mu\right)  -\left(\delta_\pf\mathcal{E}_{\mathrm{e}}(\pf, \bm u),q^\mu\right) &=& 0 \label{eq:varch2}\\
  \left( \mathbb{C}(\varphi)\left(\bm\varepsilon\left(\bm u\right)- \xi\pf\bm I \right);\bm\varepsilon(\bm v)\right) - \left( \bm f,\bm v\right) &=& 0, \label{eq:varelastic}
\end{eqnarray}
for all $\left(q^\pf, q^\mu, \bm v\right)\in H^1(\Omega)\times H^1(\Omega)\times \left(H^1_0(\Omega) \right)^d,$ and almost all $t\in [0,T]$.

As notation for the discrete equations, let $\tau$ be a uniform time-step size, defined by $\tau:=\frac{T}{N}$, where $N$ is the number of time steps. Moreover, the index $n$ will refer to the time step, $h$ the mesh diameter, and $i$ the iteration number. Let $Q_h\subseteq H^1(\Omega)$ and $\bm V_h\subseteq \left(H^1_0(\Omega)\right)^d$ be conforming 
finite element function spaces, where $Q_h$ is the solution space for phase-field and chemical potential, and $\bm V_h$ is the solution space for the displacement. 
Furthermore, we define $Q_{h,0} = \left\{q_h \in Q_h: \int_\Omega q_h\; dx = 0 \right\}$, and consider the dual space of $\left(Q_{h,0},\|\cdot\|_{h,m}\right)$ where $\|q_h\|_{h,m}:=\|m^{\frac{1}{2}}\nabla q_h\|_{L^2(\Omega)}$ 
as $Q^{*}_{h,m}$ with canonical dual norm $\|\cdot\|_ {Q_{h,m}^*}$. Notice that the space $Q_{h,m}^*$ is a discrete superspace of $H^{-1}(\Omega)$.

Due to the Lax-Milgram lemma there exists a unique $v_h\in Q_{h,0}$ for all $s_h\in Q_{h,m}^{*}$ such that 
\begin{equation}\label{eq:duality-h1}
    \langle s_h,q_h\rangle = \left(m\nabla v_h,\nabla q_h\right), \quad \forall q_h\in Q_{h,0}.
\end{equation}
Thereby, we have 
\begin{equation}\label{eq:QdualNorm}
\|s_h\|_{Q^{*}_{h,m}}:= \sup_{\substack{q_h\in Q_{h,0}\\ \|q_h\|_{h,m}\neq 0}}\frac{\langle s_h, q_h\rangle}{\|q_h\|_{h,m}} = \sup_{\substack{q_h\in Q_{h,0}\\\|q_h\|_{h,m}\neq0}}\frac{ \left(m\nabla v_h,\nabla q_h\right)}{\|m^{\frac{1}{2}}\nabla q_h\|_{L^2(\Omega)}} = \|m^{\frac{1}{2}}\nabla v_h\|_{L^2(\Omega)},
\end{equation}
where $v_h$ satisfies \eqref{eq:duality-h1}. Moreover, we identify the ${Q^{*}_{h,m}}$ inner-product for $s_h,l_h\in Q_{h,0}$ as 
\begin{equation}
    \langle s_h,l_h\rangle_{Q^{*}_{h,m}}:=(s_h,v_h)
\end{equation}
where $v_h\in Q_{h}$ is a solution to the variational equation
\begin{equation}\label{eq:preremark}
    (l_h,q_h) = (m\nabla v_h, \nabla q_h), \quad\forall q_h\in Q_{h,0}.
\end{equation}
We then have that 
\begin{equation}
    \langle s_h, s_h\rangle^{\frac{1}{2}}_{Q_{h,m}^{*}} = (s_h,r_h)^{\frac{1}{2}} = (m\nabla r_h, \nabla r_h)^{\frac{1}{2}} = \|m^{\frac{1}{2}}\nabla r_h\|_{L^2(\Omega)} = \|s_h\|_{Q_{h,m}^{*}}
\end{equation}
where $r_h\in Q_h$ satisfies $(s_h,q_h) = (m\nabla r_h, \nabla q_h)$ for all $q_h\in Q_{h,0}$.
\begin{remark}\label{rem:canonical}
Notice that, as $l_h\in Q_{h,0}$, equation \eqref{eq:preremark} holds for all $q_h\in Q_h$, and uniqueness of $v_h$ can be imposed by prescribing its mean. Choosing different values for the mean of $v_h$ does not alter the value of the inner-product $(s_h,v_h)$ as $s_h\in Q_{h,0}$.
\end{remark}

\subsection{Solution strategy for Cahn-Larch\'e with phase-field-independent elasticity tensor}\label{sec:homo}
Here, we present a robust solution strategy for the Cahn-Larch\'e equations in the special case where the elasticity tensor is independent of the phase-field, \eqref{eq:homoch1}--\eqref{eq:homoelasticity}. First, we discretize the equations by the convex-concave splitting of the double-well potential (A1), i.e., we evaluate the convex part implicitly in time and the expansive part explicitly to make the discrete system more suitable for linearization techniques. Moreover, we show that the discrete system of equations are equivalent to a minimization problem and utilize its structure to show unconditional gradient stability of the discretization (the free energy of the system does not increase without the presence of external contributions). Then, we prove convergence of alternating minimization applied to the minimization problem.

\subsubsection{Discrete system of equations}
Using the convex-concave splitting method in time for the double-well potential (A1), and evaluating other terms implicitly, we get the discretized (in time and space) system of equations corresponding to \eqref{eq:varch1}--\eqref{eq:varelastic} with  phase-field independent elasticity tensor as:
Given $\pf^{n-1}_h\in Q_h$, find $\pf^n_h, \mu_h^n \in Q_h$ and $\bm u^n_h \in \bm V_h$, such that 
\begin{eqnarray}
\left( \frac{\pf^n_h-\pf^{n-1}_h}{\tau},q^\pf_h \right) + \left( m \nabla\mu^n_h,\nabla q^\pf_h\right) -\left(  R^n ,q_h^\pf\right)&=&0 \label{eq:homoweakch1} \\
 \left( \mu_h^n,q_h^\mu\right)- \gamma\ell\left( \nabla\pf_h^n,\nabla q_h^\mu\right) - \frac{\gamma}{\ell}\left(  \Psi_c'(\pf_h^n) - \Psi_e'(\pf_h^{n-1}), q_h^\mu\right)  +\left( \mathbb{C}\left(\bm\varepsilon\left(\bm u^n_h\right) - \xi\pf_h^n\bm I\right); q_h^\mu\xi\bm I\right) &=& 0 \label{eq:homoweakch2}\\
  \left( \mathbb{C}\left(\bm\varepsilon\left(\bm u_h^n\right)- \xi\pf_h^n\bm I \right);\bm\varepsilon(\bm v_h)\right) - \left( \bm f^n,\bm v_h\right) &=& 0, \label{eq:homoweakelastic}
\end{eqnarray}
for all $q_h^\pf, q_h^\mu \in Q_h$, and all $\bm v_h \in \bm V_h$. Similar discretizations have been considered in \cite{graser2014numerical} for a phase-field dependent elasticity tensor, and in \cite{garcke2005CHENumerics} without a convex-concave splitting of the double-well potential.

\begin{prop}\label{prop:minimization}
The solution to the discrete problem \eqref{eq:homoweakch1}--\eqref{eq:homoweakelastic} is equivalent to the solution of the minimization problem: Given $\pf^{n-1}_h\in Q_h$, solve
\begin{eqnarray}\label{eq:minimizationhomo}
    (\pf^n_h,\bm u^n_h) = \argmin_{s_h\in \bar{Q}^n_{h}, \bm w_h\in \bm V_h} \mathcal{H}^n_\tau(s_h, \bm w_h)
\end{eqnarray}
where the admissible space for the phase-field is defined as
\begin{equation}\label{eq:admissible}
\bar{Q}_h^n := \left\{ s_h \in Q_h \,\left|\, \int_\Omega \frac{s_h - \varphi_h^{n-1}}{\tau} \, dx = \int_\Omega R^n \, dx \right. \right\}    
\end{equation}
and
$$ \mathcal{H}^n_\tau(s_h,\bm w_h) := \dfrac{\|s_h-\pf_h^{n-1}-\tau R^n\|^2_{Q_{h,m}^*}}{2\tau} + \mathcal{E}_c(s_h,\bm w_h) - \frac{\gamma}{\ell}\left(\Psi_e'(\pf^{n-1}_h),s_h\right)
 - \left( \bm f^n, \bm w_h \right),$$
where
\begin{equation*}
    \mathcal{E}_c(s_h,\bm w_h) := \int_\Omega \frac{\gamma}{\ell}\Psi_c(s_h) + \gamma \ell\frac{|\nabla s_h|^2}{2} + \frac{1}{2}\left(\bm \varepsilon(\bm w_h) - \xi s_h \bm I\right):\mathbb{C}\left(\bm \varepsilon(\bm w_h) - \xi s_h \bm I\right)\; dx.
\end{equation*}
\end{prop}
\begin{proof}
We derive the optimality conditions of the minimization problem which are similar to \eqref{eq:homoweakch1}--\eqref{eq:homoweakelastic}, but over restricted spaces. By employing canonical extensions, we establish the equivalence. Let $\delta_\pf\mathcal{H}^n_\tau$ and $\delta_{\bm u}\mathcal{H}^n_\tau$ represent the variational derivatives with respect to the first and second argument of the potential $\mathcal{H}^n_\tau$ respectively. Then the optimality conditions to the minimization problem \eqref{eq:minimizationhomo} reads: Find $\pf_h^n, \bm u_h^n \in \bar{Q}^n_h\times \bm V_h$ such that
\begin{eqnarray}\label{eq:opt1}
   0 =  \langle \delta_\pf\mathcal{H}_\tau^n(\pf^n_h,\bm u^n_h),q_h\rangle &=& \left\langle \frac{\pf_h^n-\pf^{n-1}_h}{\tau}-R^n,q_h\right\rangle_{Q_{h,m}^{*}} + \left( \delta_\pf \mathcal{E}_c(\pf_h^n,\bm u_h^n)-\frac{\gamma}{\ell}\Psi_e'(\pf_h^{n-1}),q_h\right)\\\label{eq:minimizationu}
   0 = \langle \delta_{\bm u}\mathcal{H}_\tau^n(\pf^n_h,\bm u^n_h),\bm w_h\rangle &=& \left( \delta_{\bm \varepsilon(\bm u)}\mathcal{E}_c(\pf^{n}_h,\bm u_h^n);\bm \varepsilon(\bm w_h)\right)- \left( \bm f^n,\bm w_h\right),
\end{eqnarray}
for all $q_h\in Q_{h,0}$ and $\bm w_h\in \bm V_h$ where
\begin{equation*}
    \delta_\pf \mathcal{E}_c(\pf_h^n,\bm u_h^n) = \frac{\gamma}{\ell}\Psi_c'(\pf_h^n) - \gamma\ell\Delta \pf_h^n -\xi\bm I:\mathbb{C}\left(\bm \varepsilon(\bm u_h^n)- \xi \pf_h^n\bm I  \right)
\end{equation*}
and 
\begin{equation*}
    \delta_{\bm \varepsilon(\bm u)}\mathcal{E}_c(\pf_h^n, \bm u_h^n) = \mathbb{C}\big(\bm\varepsilon(\bm u_h^n)- \xi\pf_h^n\bm I \big).
\end{equation*}
Using the definition of $\langle\cdot,\cdot\rangle_{Q_{h,m}^*}$ equation \eqref{eq:opt1} is equivalent to
\begin{equation}\label{eq:minimizationmu}
   0 = \left( -\mu_h^n,q_h\right) + \left( \delta_\pf \mathcal{E}_c(\pf_h^n,\bm u_h^n)-\frac{\gamma}{\ell}\Psi_e'(\pf_h^{n-1}),q_h\right), \quad\forall q_h\in Q_{h,0}
\end{equation}
where $\mu_h^n$ is the solution to the problem 
\begin{equation}\label{eq:minimizaitonpf}
-(m\nabla \mu_h^n, \nabla l_h) = \left(\frac{\pf_h^n-\pf^{n-1}_h}{\tau}-R^n,l_h\right), \quad\forall l_h\in Q_{h,0},
\end{equation}
with mean fixed as
\begin{equation}\label{eq:restrictionmu}
\int_\Omega \mu_h^n
\; dx = \int_\Omega  \delta_\pf \mathcal{E}_c(\pf_h^n,\bm u_h^n)-\frac{\gamma}{\ell}\Psi_e'(\pf_h^{n-1})\; dx,
\end{equation}
in accordance with Remark~\ref{rem:canonical}. The constraint $\pf_h^n \in \bar{Q}_h^n$ and \eqref{eq:minimizaitonpf} are equivalent to requiring that equality \eqref{eq:minimizaitonpf} holds for all $l_h\in Q_h$. Due to \eqref{eq:restrictionmu}, equation \eqref{eq:minimizationmu} holds for all $q_h\in Q_h$, and we have that the solutions to \eqref{eq:minimizationu}, \eqref{eq:minimizationmu} and \eqref{eq:minimizaitonpf} are equivalent to the solutions of the discrete problem \eqref{eq:homoweakch1}--\eqref{eq:homoweakelastic}.
\end{proof}

\begin{remark}[Affine structure of the admissible set]
The admissible set for the phase-field in the optimization problem~\eqref{eq:minimizationhomo}, $\bar{Q}_h^n$, is an affine space. For any two $s_h^1,\ s_h^2\in \bar{Q}_h^n$ it holds that $s_h^1 - s_h^2 \in Q_{h,0}$.
\end{remark}

\begin{theorem}
    The discretization scheme \eqref{eq:homoweakch1}--\eqref{eq:homoweakelastic} is unconditionally gradient stable, i.e., the free energy 
    \begin{equation*}
        \mathcal{E}(\pf, \bm u) = \int_\Omega \gamma\left(\frac{1}{\ell}\Psi(\pf) + \frac{\ell}{2}|\nabla \pf|^2\right) + \frac{1}{2}\left(\bm \varepsilon(\bm u) - \xi \pf \bm I\right):\mathbb{C}\left(\bm \varepsilon(\bm u) - \xi \pf \bm I\right)\; dx
    \end{equation*}
    dissipates over the time-steps assuming the absence of external contributions ($R = 0$ and $\bm f = 0$). 
\end{theorem}
\begin{proof}
Exploiting the equivalence between the discrete system of equations \eqref{eq:homoweakch1}--\eqref{eq:homoweakelastic} and the minimization problem in Proposition~\ref{prop:minimization}, we get that
\begin{eqnarray*}
\mathcal{H}^n_\tau(\pf^n_h, \bm u^n_h)-\mathcal{H}^n_\tau(\pf^{n-1}_h, \bm u^{n-1}_h) &\leq& 0,
\end{eqnarray*}
due to the fact that $\pf^{n-1}\in \bar{Q}_h^n$ when $R=0$. It follows that
\begin{eqnarray*}
\dfrac{\|\pf_h^n-\pf_h^{n-1}\|^2_{Q^*_{h,m}}}{2\tau} + \mathcal{E}_c(\pf_h^n,\bm u_h^n) - \frac{\gamma}{\ell}\left(\Psi_e'(\pf^{n-1}_h),\pf_h^n\right) 
-\Bigg[\mathcal{E}_c(\pf_h^{n-1},\bm u_h^{n-1}) - \frac{\gamma}{\ell}\left(\Psi_e'(\pf^{n-1}_h),\pf_h^{n-1}\right) \Bigg] &\leq& 0,
\end{eqnarray*}
and by rearrangement and application of the convexity of $\Psi_e$ we get
\begin{equation*}
    \Psi_e(\pf^n_h)- \Psi_e(\pf^{n-1}_h) \geq \Psi_e'(\pf^{n-1}_h)(\pf^{n}_h - \pf_h^{n-1}).
\end{equation*}
Recalling that $\Psi(s) = \Psi_c(s) - \Psi_e(s)$ we get the inequality
\begin{equation*}
    \dfrac{\|\pf_h^n-\pf_h^{n-1}\|^2_{Q^*_{h,m}}}{2\tau} + \mathcal{E}(\pf_h^n,\bm u_h^n) - \mathcal{E}(\pf_h^{n-1},\bm u_h^{n-1})\leq 0.
\end{equation*}
Hence,
\begin{equation*}
    \mathcal{E}(\pf_h^n,\bm u_h^n) \leq  \mathcal{E}(\pf_h^{n-1},\bm u_h^{n-1})
\end{equation*}
for all $\tau$ and $n$.
\end{proof}

\subsubsection{Alternating minimization for the Cahn-Larch\'e equations with phase-field-independent elasticity tensor}\label{sec:am-homo}
There exists several ways to solve the nonlinear discrete system of equations \eqref{eq:homoweakch1}--\eqref{eq:homoweakelastic}, and due to the convexity of the related minimization problem (see Proposition~\ref{prop:minimization}) we expect the Newton method to be a viable and efficient choice. However, we propose here to solve the system with an alternating minimization method. The  main benefit of this is that it allows for the use of readily available solvers, as it corresponds to solving a Cahn-Hilliard equation and an elasticity equation subsequently. In each time step we initialize the solver with the solution at the previous time step
\begin{equation*}
    \pf_h^{n,0} = \pf_h^{n-1},\quad\mathrm{and}\quad \bm u_h^{n,0} = \bm u_h^{n-1},
\end{equation*}
and minimize the potential $\mathcal{H}^n_\tau$ sequentially
\begin{eqnarray}\label{eq:min1} 
    \pf^{n,i}_h &=& \argmin_{s_h\in \bar{Q}^n_h}\mathcal{H}^n_\tau(s_h,\bm u_h^{n,i-1}),\\
\label{eq:min2}
\bm u^{n,i}_h &=& \argmin_{\bm w_h\in \bm V_h}\mathcal{H}^n_\tau(\pf_h^{n,i},\bm w_h)
\end{eqnarray}
where $i$ is the iteration index. The corresponding variational system of equations in the $i$-th iteration reads: Given $(\pf_h^{n-1},\bm u_h^{n,i-1})\in Q_h\times \bm V_h$, find $(\pf_h^{n,i}, \mu_h^{n,i},\bm u_h^{n,i})\in Q_h\times Q_h\times \bm V_h$ such that
\begin{eqnarray}
\left( \frac{\pf^{n,i}_h-\pf^{n-1}_h}{\tau},q^\pf_h \right) + \left( m \nabla\mu_h^{n,i},\nabla q^\pf_h\right) -\left( R^n ,q_h^\pf\right)&=&0 \label{eq:homosplitweakch1} \\
 \left( \mu_h^{n,i},q_h^\mu\right) - \gamma\ell\left( \nabla\pf_h^{n,i},\nabla q_h^\mu\right) - \frac{\gamma}{\ell}\left(  \Psi_c'\left(\pf_h^{n,i}\right) - \Psi_e'\left(\pf_h^{n-1}\right), q_h^\mu\right)  &&\nonumber\\+\left( \mathbb{C}\left(\bm\varepsilon\left(\bm u^{n,i-1}_h\right) - \xi\pf_h^{n,i}\bm I\right); q_h^\mu\xi\bm I\right) &=& 0 \label{eq:homosplitweakch2}\\
  \left( \mathbb{C}\left(\bm\varepsilon\left(\bm u_h^{n,i}\right)- \xi\pf_h^{n,i}\bm I \right);\bm\varepsilon(\bm v_h)\right) - \left( \bm f^n,\bm v_h\right) &=& 0\label{eq:homosplitweakelasticity}
\end{eqnarray}
for all $(q^\pf_h, q^\mu_h,\bm v_h)\in Q_h\times Q_h\times \bm V_h$. Here, the space $Q_h$ appears in the discrete system instead of $\bar{Q}_h^n$ due to the same argumentation as in the proof of Proposition~\ref{prop:minimization}.

\begin{remark}
The Cahn-Hilliard subsystem \eqref{eq:homosplitweakch1}--\eqref{eq:homosplitweakch2} is still nonlinear due to $\Psi'_c(\pf^{n,i})$. In this work, we solve it with the Newton method which is known to converge for this problem \cite{guillenNewton}.
\end{remark}

We apply the abstract theory available in \cite{jakubAM} to prove that the alternating minimization algorithm converges and summarize the appropriate result as a lemma (using the notation of the present article):
\begin{lemma}\label{lem:Jakub}
Assume that there exist norms $\|(\cdot,\cdot)\|:Q_{h,0}\times\bm V_h\rightarrow \mathbb{R}^+$, $\|\cdot\|_\mathrm{ch}:Q_{h,0}\rightarrow\mathbb{R}^+$ and $\|\cdot\|_\mathrm{e}:\bm V_h\rightarrow\mathbb{R}^+$, related by the inequalities
\begin{equation}\label{eq:normrelation}
     \|(s_h,\bm w_h)\|^2 \geq \beta_\mathrm{ch}\|s_h\|_\mathrm{ch}^2,\quad\mathrm{and}\quad \|(s_h,\bm w_h)\|^2\geq \beta_\mathrm{e}\|\bm w_h\|_\mathrm{e}^2, \quad\forall (s_h,\bm w_h)\in Q_{h,0}\times \bm V_h,
\end{equation}
for some $\beta_{\mathrm{ch}},\beta_{\mathrm{e}}\geq 0$, and let the potential $\mathcal{H}:\bar{Q}_h^n\times \bm V_h\rightarrow \mathbb{R}$ be given. If
\begin{itemize}
    \item $\mathcal{H}$ is convex with respect to the norm $\|(\cdot,\cdot)\|$ with convexity constant $\sigma\geq 0$, i.e.,
    \begin{equation}\label{eq:coercive}
        \left\langle \delta \mathcal{H}\left(s_h^1,\bm w_h^1\right)-\delta \mathcal{H}\left(s_h^2,\bm w_h^2\right),\left(s_h^1-s_h^2, \bm w_h^1-\bm w_h^2\right)\right\rangle \geq \sigma \left\|\left(s_h^1-s_h^2, \bm w_h^1- \bm w_h^2)\right)\right\|^2,
    \end{equation}
    for all $ \left(s_h^1, s_h^2, \bm w_h^1, \bm w_h^2\right)\in \bar{Q}^n_{h}\times \bar{Q}^n_{h} \times \bm V_h\times\bm V_h,$
\end{itemize}
 and
\begin{itemize}
    \item the variational derivatives of $\mathcal{H}$ with respect to the first and second arguments are Lipschitz continuous in the norm $\|\cdot\|_\mathrm{ch}$ with constant $L_\mathrm{ch}$ and $\|\cdot\|_\mathrm{e}$ with constant $L_\mathrm{e}$, respectively, i.e., there exist $L_\mathrm{ch}>0$, $L_\mathrm{e}>0$ such that
    \begin{equation}\label{eq:lipschitzCH}
        \left\langle\delta_\pf \mathcal{H}\left(s_h^1,\bm w_h\right)- \delta_\pf \mathcal{H}\left(s_h^2,\bm w_h\right),s_h^1-s_h^2\right\rangle \leq L_\mathrm{ch}\left\|s_h^1-s_h^2\right\|^2_ \mathrm{ch},\quad \forall \left(s_h^1,s_h^2,\bm w_h\right)\in \bar{Q}_{h}^n\times \bar{Q}_{h}^n\times\bm V_h,
    \end{equation}
    and
      \begin{equation}\label{eq:lipschitzE}
        \left\langle\delta_{\bm u} \mathcal{H}\left(s_h,\bm w_h^1\right)- \delta_{\bm u} \mathcal{H}\left(s_h,\bm w_h^2\right),\bm w_h^1-\bm w_h^2\right\rangle \leq L_\mathrm{e}\left\|\bm w_h^1-\bm w_h^2\right\|^2_ \mathrm{e},\quad \forall \left(\bm w_h^1,\bm w_h^2,s_h\right)\in \bm V_h\times \bm V_h\times \bar{Q}^n_{h},
    \end{equation}
\end{itemize}
then the alternating minimization scheme (as proposed in \eqref{eq:min1}--\eqref{eq:min2} with $\mathcal{H}^n_\tau = \mathcal{H}$) converges in the sense that
\begin{equation*}
    \mathcal{H}\left(\pf_h^{n,i},\bm u_h^{n,i}\right)-\mathcal{H}\left(\pf_h^{n},\bm u_h^{n}\right)\leq\left(1-\frac{\sigma\beta_\mathrm{ch}}{L_\mathrm{ch}}\right)\left(1-\frac{\sigma\beta_\mathrm{e}}{L_\mathrm{e}}\right)
    \left(\mathcal{H}\left(\pf_h^{n,i-1},\bm u_h^{n,i-1}\right)-\mathcal{H}\left(\pf_h^{n},\bm u_h^{n}\right)\right),
\end{equation*}
where $\left(\pf_h^{n},\bm u_h^{n}\right)\in \bar{Q}_h^n\times \bm V_h$ is the minimizer of $\mathcal{H}$.
\end{lemma}
\begin{remark}
Notice that $\frac{\sigma\beta_\mathrm{ch}}{L_\mathrm{ch}}\leq 1$ and $\frac{\sigma\beta_\mathrm{e}}{L_\mathrm{e}}\leq 1$ due to \eqref{eq:normrelation}--\eqref{eq:lipschitzE}.
\end{remark}
We are also going to take advantage of the following inverse inequality:
\begin{lemma}\label{lem:inverseinequality}
There exists a constant $C_\mathrm{inv}>0$ such that 
\begin{equation*}
    C_\mathrm{inv}h^{-1}\|s_h\|_{Q^*_{m,h}}\geq\|s_h\|_{L^2(\Omega)},
\end{equation*}
for all $s_h\in Q_{h,0}$.
\end{lemma}
\begin{proof}
From standard finite element text books, e.g., Theorem 4.5.11 in \cite{brennerFEM}, one can find the inverse inequality
\begin{equation}\label{eq:inverseineq}
    \|s_h\|_{H^1(\Omega)} \leq \tilde{C}h^{-1}\|s_h\|_{L^2(\Omega)},
\end{equation}
for some $\tilde{C}>0$. By the definition of the $Q^*_{h,m}$-norm \eqref{eq:QdualNorm} we have
for $s_h\in Q_{h,0}$ and $\|s_h\|_{h,m}\neq 0$ 
\begin{equation*}
    \|s_h\|_{Q^*_{h,m}} \geq \frac{\left\langle s_h,s_h\right\rangle}{\|m^\frac{1}{2}\nabla s_h\|_{L^2(\Omega)}},
\end{equation*}
which implies
\begin{equation*}
    m^\frac{1}{2}\|s_h\|_{H^1(\Omega)}\|s_h\|_{Q^*_{m,h}} \geq \|s_h\|^2_{L^2(\Omega)}.
\end{equation*}
Using \eqref{eq:inverseineq} we get by choosing $C_\mathrm{inv} = \tilde{C}m^\frac{1}{2}$ the desired inequality
\begin{equation*}
    C_\mathrm{inv}h^{-1}\|s_h\|_{L^2(\Omega)}\|s_h\|_{Q^*_{h,m}} \geq \|s_h\|^2_{L^2(\Omega)}.
\end{equation*}

\end{proof}

\begin{theorem}\label{thm:am}
    The alternating minimization algorithm \eqref{eq:min1}--\eqref{eq:min2} converges linearly in the sense that
    \begin{equation}\label{eq:convrate}
    \mathcal{H}^n_\tau\left(\pf_h^{n,i},\bm u_h^{n,i}\right)-\mathcal{H}^n_\tau\left(\pf_h^{n},\bm u_h^{n}\right)\leq\left(1-\frac{\beta_\mathrm{ch}}{L_\mathrm{ch}}\right)\left(1-\beta_\mathrm{e}\right)
    \left(\mathcal{H}^n_\tau\left(\pf_h^{n,i-1},\bm u_h^{n,i-1}\right)-\mathcal{H}^n_\tau\left(\pf_h^{n},\bm u_h^{n}\right)\right),
\end{equation}
where $\beta_\mathrm{ch} =  \beta_\mathrm{e}= 1-\left(\frac{h^2}{\tau C_\mathrm{inv}^2\xi^2\bm I:\mathbb{C}\bm I}+\frac{\gamma\ell}{C_\Omega^2\xi^2\bm I:\mathbb{C}\bm I}+1\right)^{-1}$, and $L_\mathrm{ch} = 1+ L_\Psi\left(\frac{h^2}{\tau C_\mathrm{inv}^2}+\frac{\gamma\ell}{C_\Omega^2}+\xi^2\bm I:\mathbb{C}\bm I\right)^{-1}$.

\end{theorem}
\begin{proof}
We apply Lemma~\ref{lem:Jakub}. Let $\mathcal{H} = \mathcal{H}_\tau^n$, $\bar{Q}_h = \bar{Q}_h^n$, and define the norms
\begin{eqnarray*}
    \|(s_h,\bm w_h)\|^2&:=& \frac{\|s_h\|^2_{Q^*_{h,m}}}{\tau} + \gamma\ell\|\nabla s_h\|^2_{L^2(\Omega)} +  \left(\mathbb{C}\left(\bm\varepsilon(\bm w_h) - \xi s_h\bm I\right);\bm\varepsilon(\bm w_h) - \xi s_h\bm I\right),\\
    \|s_h\|_{\mathrm{ch}}^2&:=& \frac{\|s_h\|^2_{Q^*_{h,m}}}{\tau} + \gamma\ell\|\nabla s_h\|^2_{L^2(\Omega)} + \xi^2\bm I:\mathbb{C}\bm I \|s_h\|^2_{L^2(\Omega)},\\
    \|\bm w_h\|^2_\mathrm{e} &:=& \left( \mathbb{C}\bm \varepsilon(\bm w_h);\bm\varepsilon(\bm w_h)\right),
\end{eqnarray*}
for $(s_h,\bm w_h)\in Q_{h,0}\times \bm V_h$.
Notice that $\|(\cdot,\cdot)\|$ and $\|\cdot\|_\mathrm{e}$ are norms due to \eqref{eq:posdefC}.

\vspace{0.2cm}

\noindent
\textit{Relation~\eqref{eq:normrelation} between norms.} We have that for $(s_h, \bm w_h)\in Q_{h,0}\times \bm V_h$
\begin{eqnarray}
    \|(s_h,\bm w_h)\|^2&=&\frac{\|s_h\|^2_{Q^*_{h,m}}}{\tau} + \gamma\ell\|\nabla s_h\|^2_{L^2(\Omega)} +  \left(\mathbb{C}\left(\bm\varepsilon(\bm w_h) \right);\bm\varepsilon(\bm w_h)\right)\label{eq:fullnorm}\\
     &&+ \xi^2\bm I:\mathbb{C}\bm I\|s_h\|^2_{L^2(\Omega)} -2\left(\mathbb{C}\left(\bm\varepsilon(\bm w_h)\right);\xi s_h\bm I \right)\nonumber,
\end{eqnarray}
and by the Cauchy-Schwarz' inequality \eqref{eq:tensorCS} and  Young's inequality on the last term we obtain
\begin{eqnarray*}
  2\left(\mathbb{C}\left(\bm\varepsilon(\bm w_h)\right);\xi s_h\bm I \right)&\leq& \delta\left( \mathbb{C}\bm \varepsilon(\bm w_h);\bm \varepsilon(\bm w_h)\right) + \frac{k_1\xi^2\bm I:\mathbb{C}\bm I}{\delta}\|s_h\|_{L^2(\Omega)}^2\\&& + \frac{k_2\xi^2\bm I:\mathbb{C}\bm I}{\delta}\|s_h\|_{L^2(\Omega)}^2 + \frac{k_3\xi^2\bm I:\mathbb{C}\bm I}{\delta}\|s_h\|_{L^2(\Omega)}^2
\end{eqnarray*}
where $1\geq k_i\geq0$, $k_1+k_2+k_3 = 1$, and $\delta>0$ are free to be chosen. Using Lemma~\ref{lem:inverseinequality} and the Poincar\'e inequality, with constant $C_\Omega$, we get
\begin{eqnarray*}
  2\left(\mathbb{C}\left(\bm\varepsilon(\bm w_h)\right);\xi s_h\bm I \right)&\leq& \delta\left( \mathbb{C}\bm \varepsilon(\bm w_h);\bm \varepsilon(\bm w_h)\right) + \frac{k_1C_{\mathrm{inv}}^2h^{-2}\xi^2\bm I:\mathbb{C}\bm I}{\delta}\|s_h\|_{Q^*_{h,m}}^2 \\&&+\frac{k_2C_\Omega^2\xi^2\bm I:\mathbb{C}\bm I}{\delta}\|\nabla s_h\|_{L^2(\Omega)}^2 +\frac{k_3\xi^2\bm I:\mathbb{C}\bm I}{\delta}\|s_h\|_{L^2(\Omega)}^2.\nonumber
\end{eqnarray*}
Hence, we have from \eqref{eq:fullnorm} that
\begin{eqnarray}\label{eq:doublenormbound}
     \|(s_h,\bm w_h)\|^2 &\geq& (1-\delta)\left( \mathbb{C}\bm \varepsilon(\bm w_h);\bm \varepsilon(\bm w_h)\right)+ \left(\frac{1}{\tau}-  \frac{k_1C_{\mathrm{inv}}^2h^{-2}\xi^2\bm I:\mathbb{C}\bm I}{\delta}\right)\|s_h\|^2_{Q^*_{h,m}}\\&& + \left(\gamma\ell -\frac{k_2C_\Omega^2\xi^2\bm I:\mathbb{C}\bm I}{\delta}\right)\|\nabla s_h\|^2_{L^2(\Omega)} + \left(1-\frac{k_3}{\delta}\right)\xi^2\bm I:\mathbb{C}\bm I\|s_h\|^2_{L^2(\Omega)}.\nonumber
\end{eqnarray}
Choosing $\delta=1$, $\beta_\mathrm{ch} = 1-\left(\frac{h^2}{\tau C_\mathrm{inv}^2\xi^2\bm I:\mathbb{C}\bm I}+\frac{\gamma\ell}{C_\Omega^2\xi^2\bm I:\mathbb{C}\bm I}+1\right)^{-1}$, $k_1 = (1-\beta_\mathrm{ch})\frac{h^2}{\tau C_\mathrm{inv}^2\xi^2\bm I:\mathbb{C}\bm I}$, $k_2 =(1-\beta_\mathrm{ch})\frac{\gamma\ell}{C_\Omega^2\xi^2\bm I:\mathbb{C}\bm I}$, and $k_3 = 1-\beta_\mathrm{ch}$ we get the desired bound
\begin{equation*}
     \|(s_h,\bm w_h)\|^2 \geq \beta_\mathrm{ch}\|s_h\|^2_\mathrm{ch},\quad \forall (s_h,\bm w_h) \in Q_{h,0}\times \bm V_h.
\end{equation*}
Choosing now $\delta = \left(\frac{h^2}{\tau C_\mathrm{inv}^2\xi^2\bm I:\mathbb{C}\bm I}+\frac{\gamma\ell}{C_\Omega^2\xi^2\bm I:\mathbb{C}\bm I}+1\right)^{-1}$, $k_1 = \frac{\delta h^2}{\tau C_\mathrm{inv}^2\xi^2\bm I:\mathbb{C}\bm I}$, $k_2 = \frac{\gamma\ell\delta}{C_\Omega^2\xi^2\bm I:\mathcal{C}\bm I}$, $k_3 = \delta$, and  $\beta_\mathrm{e}= 1-\delta$ in equation~\eqref{eq:doublenormbound} we obtain
\begin{equation*}
     \|(s_h,\bm w_h)\|^2 \geq \beta_\mathrm{e}\|\bm w_h\|^2_\mathrm{e},\quad \forall (s_h,\bm w_h) \in Q_{h,0}\times \bm V_h.
\end{equation*}

\vspace{0.2cm}

\noindent
\textit{Strong convexity.}
By assumption (A2)
\begin{eqnarray*}
 &\left\langle \delta \mathcal{H}^n_\tau(s_h^1,\bm w_h^1)-\delta \mathcal{H}^n_\tau(s_h^2,\bm w_h^2),\left(s_h^1-s_h^2, \bm w_h^1-\bm w_h^2\right)\right\rangle\\
 &=  \left\langle \delta_\pf \mathcal{H}^n_\tau(s_h^1,\bm w_h^1)-\delta_\pf \mathcal{H}^n_\tau(s_h^2,\bm w_h^2),s_h^1-s_h^2\right\rangle +  \left\langle \delta_{\bm u} \mathcal{H}^n_\tau(s_h^1,\bm w_h^1)-\delta_{\bm u} \mathcal{H}^n_\tau(s_h^2,\bm w_h^2),\bm w_h^1-\bm w_h^2\right\rangle\\
 &= \left\|\left(s_h^1-s_h^2,\bm w_h^1-\bm w_h^2\right)\right\|^2 + \frac{\gamma}{\ell}\left( \Psi_c'(s_h^1)-\Psi_c'(s_h^2),s_h^1-s_h^2\right)\\
 &\geq \left\|\left(s_h^1-s_h^2,\bm w_h^1-\bm w_h^2\right)\right\|^2,
\end{eqnarray*}
for all $ \left(s_h^1, s_h^2, \bm w_h^1, \bm w_h^2\right)\in \bar{Q}_{h}^n\times \bar{Q}_{h}^n \times \bm V_h\times \bm V_h,$ we have that $\mathcal{H}_\tau^n(s_h,\bm w_h)$ is convex in $\|(s_h,\bm w_h)\|$ with convexity constant $\sigma = 1$.

\vspace{0.2cm}

\noindent
\textit{Lipschitz continuity of the partial gradients.} We have 
\begin{eqnarray*}
    \left\langle\delta_\pf \mathcal{H}_\tau^n(s_h^1,\bm w_h)- \delta_\pf \mathcal{H}_\tau^n(s_h^2,\bm w_h),s_h^1-s_h^2\right\rangle = \|s_h^1-s_h^2\|^2_ \mathrm{ch} + \frac{\gamma}{\ell}\left( \Psi_c'(s_h^1)-\Psi_c'(s_h^2),s_h^1-s_h^2\right)
\end{eqnarray*}
for all $\left(s_h^1, s_h^2, \bm w_h\right)\in \bar{Q}^n_h\times \bar{Q}^n_h\times \bm  V_h$. Assumption (A2) gives
\begin{equation*}
    \left( \Psi_c'(s_h^1)-\Psi_c'(s_h^2),s_h^1-s_h^2\right)
    \leq L_{\Psi_c}\|s_h^1-s_h^2\|^2_{L^2(\Omega)}
\end{equation*}
and by Lemma~\ref{lem:inverseinequality} we get
\begin{eqnarray*}
    \left\langle\delta_\pf \mathcal{H}_\tau^n(s_h^1,\bm w_h)- \delta_\pf \mathcal{H}_\tau^n(s_h^2,\bm w_h),s_h^1-s_h^2\right\rangle \leq L_\mathrm{ch}\|s_h^1-s_h^2\|^2_\mathrm{ch},
\end{eqnarray*}
where $L_\mathrm{ch} = 1+ L_\Psi\left(\frac{h^2}{\tau C_\mathrm{inv}^2}+\frac{\gamma\ell}{C_\Omega^2}+\xi^2\bm I:\mathbb{C}\bm I\right)^{-1}$. Finally, $\delta_{\bm u}\mathcal{H}^n_\tau$ is Lipschitz continuous with respect to $\|\cdot\|_\mathrm{e}$ with constant $L_\mathrm{e}=1$, since
\begin{equation*}
      \left\langle\delta_{\bm u} \mathcal{H}_\tau^n(s_h,\bm w_h^1)- \delta_{\bm u} \mathcal{H}_\tau^n(s_h,\bm w_h^2),\bm w_h^1-\bm w_h^2\right\rangle = \|\bm w_h^1-\bm w_h^2\|^2_ \mathrm{e},\quad\forall \left(s_h,\bm w_h^1,\bm w_h^2\right)\in \bar{Q}^n_h\times \bm V_h, \times\bm V_h,
\end{equation*}
and the convergence result \eqref{eq:convrate} is obtained through Lemma~\ref{lem:Jakub}.
\end{proof}

\subsection{Solution strategy for the Cahn-Larch\'e equations with phase-field-dependent elasticity tensor}\label{sec:hetero}
When the elasticity tensor depends on the phase-field, $\mathbb{C}(\pf)$, the situation is slightly more involved because a naive implicit discretization, using the convex-concave splitting of the double-well potential $\Psi$ leads to a discrete system that is related to a nonconvex minimization problem (similar treatment as in Proposition~\ref{prop:minimization}). It reads:  Given $\pf^{n-1}_h\in Q_h$, find $\pf^n_h, \mu_h^n \in Q_h$ and $\bm u^n_h \in \bm V_h$, such that 
\begin{eqnarray}\label{eq:weakimpch1}
\left( \frac{\pf^n_h-\pf^{n-1}_h}{\tau},q^\pf_h \right) + \left( m \nabla\mu_h^n,\nabla q^\pf_h\right) -\left(  R^n ,q_h^\pf\right)&=&0,  \\\label{eq:weakimpch2}
 \left( \mu_h^n,q_h^\mu\right) - \gamma\ell\left( \nabla\pf_h^n,\nabla q_h^\mu\right) - \frac{\gamma}{\ell}\left(  \Psi_c'(\pf_h^n) - \Psi_e'(\pf_h^{n-1}), q_h^\mu\right)  -\left( \delta_\pf\mathcal{E}_\mathrm{e}(\pf^n_h,\bm u^n_h), q_h^\mu\right) &=& 0, \\
  \left( \mathbb{C}\left(\pf_h^n\right)\big(\bm\varepsilon(\bm u_h^n)- \xi\pf_h^n\bm I) \big);\bm\varepsilon(\bm v_h)\right) - \left( \bm f^n,\bm v_h\right) &=& 0, \label{eq:weakimpelastic}
\end{eqnarray}
for all $(q^\pf_h, q^\mu_h,\bm v_h)\in Q_h\times Q_h\times \bm V_h$ with $\delta_\pf\mathcal{E}_\mathrm{e}(\pf_h^n,\bm u_h^n)$ from \eqref{eq:elasticenergyderivative}. To mitigate the nonconvexity of the related minimization problem one could evaluate the entire term related to the elastic energy explicitly, $\delta_\pf\mathcal{E}_\mathrm{e}(\pf^{n-1}_h,\bm u^{n-1}_h)$. Then one could show, using the same technique as in Theorem~\ref{thm:am} that an alternating minimization type method would converge. Instead, we propose a semi-implicit evaluation of the term $\delta_\pf\mathcal{E}_\mathrm{e}(\cdot,\cdot)$, which corresponds to a convex minimization problem. The discretization reads:
Given $\pf^{n-1}_h\in Q_h$, find $\pf^n_h, \mu_h^n \in Q_h$ and $\bm u^n_h \in V_h$, such that 
\begin{eqnarray}
\left( \frac{\pf^n_h-\pf^{n-1}_h}{\tau},q^\pf_h \right) + \left( m \nabla\mu_h^n,\nabla q^\pf_h\right) -\left(  R^n ,q_h^\pf\right)&=&0, \label{eq:weakch1} \\
 \left( \mu_h^n,q_h^\mu\right) - \gamma\ell\left( \nabla\pf_h^n,\nabla q_h^\mu\right)-\frac{\gamma}{\ell}\left(  \Psi_c'(\pf_h^n) - \Psi_e'(\pf_h^{n-1}), q_h^\mu\right)  -\left( \mathcal{E}_\mathrm{e,\pf}^\mathrm{si}(\pf^n_h,\bm u^n_h; \pf_h^{n-1},\bm u_h^{n-1}), q_h^\mu\right) &=& 0, \label{eq:weakch2}\\
  \left( \mathbb{C}\left(\pf_h^{n-1}\right)\big(\bm\varepsilon(\bm u_h^n)- \xi\pf_h^n\bm I) \big);\bm\varepsilon(\bm v_h)\right) - \left( \bm f^n,\bm v_h\right) &=& 0, \label{eq:weakelastic}
\end{eqnarray}
for all $(q^\pf_h, q^\mu_h,\bm v_h)\in Q_h\times Q_h\times \bm V_h$ where 
\begin{eqnarray*}
   \mathcal{E}_\mathrm{e,\pf}^\mathrm{si}(\pf^n_h,\bm u^n_h; \pf_h^{n-1},\bm u_h^{n-1}) &:=& \frac{1}{2}\left(\bm\varepsilon\left(\bm u^{n-1}_h\right)-\xi\pf^{n-1}_h\bm I\right)\mathbb{C}'\left(\pf^{n-1}_h\right)\left(\bm \varepsilon\left(\bm u_h^{n-1}\right)-\xi\pf_h^{n-1}\bm I\right)\\
   &&-\xi\bm I:\mathbb{C}(\pf^{n-1}_h)\left(\bm\varepsilon(\bm u_h^n)-\xi\pf^{n}_h\bm I\right).
\end{eqnarray*}
Notice here, that
\begin{equation*}
     \delta_\pf\mathcal{E}_\mathrm{e}(\pf,\bm u)  =  \mathcal{E}_\mathrm{e,\pf}^\mathrm{si}(\pf,\bm u, \pf,\bm u).
\end{equation*}
Analogous to Proposition~\ref{prop:minimization} we can prove that \eqref{eq:weakch1}--\eqref{eq:weakelastic} is related to a minimization problem.
\begin{prop} 
The solution to the discrete system of equation \eqref{eq:weakch1}--\eqref{eq:weakelastic} are equivalent to the solution of the minimization problem: Given $\pf^{n-1}_h, \bm u^{n-1}_h\in Q_h\times\bm V_h$ solve
\begin{eqnarray}\label{eq:minimizationhetero}
    (\pf^n_h,\bm u^n_h) = \argmin_{s_h\in \bar{Q}^n_h, \bm w_h\in \bm V_h} \mathcal{F}^{n}_\tau(s_h, \bm w_h)
\end{eqnarray}
for 
\begin{eqnarray*}
    \mathcal{F}^n_\tau(s_h,\bm w_h) &:=& \dfrac{\|s_h-\varphi_h^{n-1} - \tau R^n\|^2_{Q_{h,m}^*}}{2\tau} + \mathcal{E}^\mathrm{c}_c(s_h,\bm w_h, \pf_h^{n-1}) -\left( \mathcal{E}^\mathrm{e}_\mathrm{e}(\pf_h^{n-1}, \bm u_h^{n-1}),s_h\right)\\
    &&- \frac{\gamma}{\ell}\left(\Psi_e'(\pf^{n-1}_h),s_h\right)  - \left( \bm f^n, \bm w_h \right),
\end{eqnarray*}
where
\begin{equation*}
    \mathcal{E}_c^c(s_h,\bm w_h, \pf^{n-1}_h) := \int_\Omega \frac{\gamma}{\ell}\Psi_c(s_h) + \gamma\ell \frac{|\nabla s_h|^2}{2} + \frac{1}{2}\left(\bm \varepsilon(\bm w_h) - \xi s_h \bm I\right):\mathbb{C}\left(\pf^{n-1}_h\right)\left(\bm \varepsilon(\bm w_h) - \xi s_h \bm I\right)\; dx,
\end{equation*}
and
\begin{equation*}
    \mathcal{E}_e^e(\varphi^{n-1}_h,\bm u^{n-1}_h) :=  \frac{1}{2}\left(\bm\varepsilon\left(\bm u^{n-1}_h\right)-\xi\pf^{n-1}_h\bm I\right)\mathbb{C}'\left(\pf^{n-1}_h\right)\left(\bm \varepsilon\left(\bm u_h^{n-1}\right)-\xi\pf_h^{n-1}\bm I\right).
\end{equation*}
\end{prop}

\subsubsection{Alternating minimization for Cahn-Larch\'e with phase-field-dependent elasticity tensor}
Similarly to Section~\ref{sec:am-homo} we propose an alternating minimization algorithm, which again naturally is formulated as a block Gauss-Seidel method, to solve the discrete system of equations \eqref{eq:weakch1}--\eqref{eq:weakelastic}.  Given $(\pf_h^{n-1}, \bm u_h^{n-1}, \bm u_h^{n,i-1})\in Q_h\times \bm V_h\times \bm V_h$, find $(\pf_h^{n,i}, \mu_h^{n,i},\bm u_h^{n,i})\in Q_h\times Q_h\times \bm V_h$ such that
\begin{eqnarray}
\left( \frac{\pf^{n,i}_h-\pf^{n-1}_h}{\tau},q^\pf_h \right) + \left( m \nabla\mu_h^{n,i},\nabla q^\pf_h\right) -\left(  R^n ,q_h^\pf\right)&=&0, \label{eq:splitweakch1} \\
 \left( \mu_h^{n,i},q_h^\mu\right) - \gamma\ell\left( \nabla\pf_h^{n,i},\nabla q_h^\mu\right) - \frac{\gamma}{\ell}\left(  \Psi_c'\left(\pf_h^{n,i}\right) - \Psi_e'\left(\pf_h^{n-1}\right), q_h^\mu\right)  \nonumber\\
 +\left(\mathcal{E}_\mathrm{e,\pf}^\mathrm{si}\left(\pf^{n,i}_h,\bm u^{n,i-1}_h, \pf_h^{n-1},\bm u_h^{n-1}\right),q_h^\mu\right) &=& 0, \label{eq:splitweakch2}\\
  \left( \mathbb{C}\left(\pf_h^{n-1}\right)\left(\bm\varepsilon\left(\bm u_h^{n,i}\right)- \xi\pf_h^{n,i}\bm I \right);\bm\varepsilon(\bm v_h)\right) - \left( \bm f^n,\bm v_h\right) &=& 0,\label{eq:splitweakelasticity}
\end{eqnarray}
for all $(q^\pf_h, q^\mu_h,\bm v_h)\in Q_h\times Q_h\times \bm V_h$. 
\begin{corollary}\label{cor:am}
The alternating minimization decoupling scheme \eqref{eq:splitweakch1}--\eqref{eq:splitweakelasticity} converges in each time-step $n$, with convergence rate
    \begin{equation}\label{eq:convratehetero}
    \mathcal{F}^n_\tau\left(\pf_h^{n,i},\bm u_h^{n,i}\right)-\mathcal{F}^n_\tau\left(\pf_h^{n},\bm u_h^{n}\right)\leq\left(1-\frac{\beta_\mathrm{ch}}{L_\mathrm{ch}}\right)\left(1-\beta_\mathrm{e}\right)
    \left(\mathcal{F}^n_\tau\left(\pf_h^{n,i-1},\bm u_h^{n,i-1}\right)-\mathcal{F}^n_\tau\left(\pf_h^{n},\bm u_h^{n}\right)\right),
\end{equation}
where $\beta_\mathrm{ch} =  \beta_\mathrm{e}= 1-\left(\frac{h^2}{\tau C_\mathrm{inv}^2\xi^2\bm I:\bm IC_\mathbb{C}}+\frac{\gamma\ell}{C_\Omega^2\xi^2\bm I:\bm I C_\mathbb{C}}+1\right)^{-1}$, and $L_\mathrm{ch} = 1+ L_\Psi\left(\frac{h^2}{\tau C_\mathrm{inv}^2}+\frac{\gamma\ell}{C_\Omega^2}+\xi^2\bm I:\bm I c_\mathbb{C}\right)^{-1}$.
\end{corollary}
\begin{proof}
This proof is analogous that of Theorem~\ref{thm:am}. Simply replace $\mathbb{C}$ with $\mathbb{C}(\pf_h^{n-1})$ and apply the bounds from assumption (A2).
\end{proof}
\begin{remark}
Notice that as the discrete system of equations \eqref{eq:weakch1}--\eqref{eq:weakelastic} corresponds to a convex minimization problem, we also expect a Newton-type solver to be rather robust, and have a higher convergence rate than the alternating minimization method. 
\end{remark}

\section{Numerical experiments}\label{sec:numerics}
In this section, we present experiments to numerically investigate the performance and robustness of both the Newton method and alternating minimization applied to the semi-implicit time-discretized Cahn-Larch\'e equations \eqref{eq:weakch1}--\eqref{eq:weakelastic} compared with applying them to the implicit-in-time discretizaton \eqref{eq:weakimpch1}--\eqref{eq:weakimpelastic}. In all numerical experiments, the unit square in two spatial dimensions with a quadrilateral mesh is considered, and we apply bilinear conforming finite elements to all subproblems; phase-field, potential, and displacement.

When the elasticity tensor depends on the phase-field it is through the $C^1$ interpolation function 
\begin{equation}
    \pi(\pf) = 
    \begin{cases} 
    0,\quad &\pf<-1\\
    \frac{1}{4}\left(-\pf^3 + 3\pf +2 \right), \quad &\pf\in [-1,1]\\
    1,\quad &\pf>1
    \end{cases},
\end{equation}
and the relation $\mathbb{C}(\pf) = \mathbb{C}_{-1} + \pi(\pf)\left(\mathbb{C}_{1}-\mathbb{C}_{-1}\right)$, where $\mathbb{C}_{-1}$ and $\mathbb{C}_{1}$ are the elasticity tensors corresponding to the pure phases at $\varphi = -1$ and $\varphi=1$, respectively.

Four different solution strategies to the Cahn-Larch\'e equations are tested. For the discrete system \eqref{eq:weakimpch1}--\eqref{eq:weakimpelastic} we test both the monolithic Newton method (marked by "Imp. Mono." in figure legends) and a staggered solution scheme, solving the Cahn-Hilliard subsystem \eqref{eq:weakimpch1}--\eqref{eq:weakimpch2} and the elasticity subsystem \eqref{eq:weakimpelastic} sequentially (marked by "Imp. Split." in figure legends). The same is done for the discrete system \eqref{eq:weakch1}--\eqref{eq:weakelastic} and mark the monolithic Newton method as "Semi-Imp. Mono." and the alternating minimization method \eqref{eq:splitweakch1}--\eqref{eq:splitweakelasticity} as "Semi-Imp. Split.". For both the monolithic and the decoupling solvers, the iterative procedures are terminated when the absolute and relative residuals and increments (iteration $i-1$ subtracted from iteration $i$), in the $L^2(\Omega)$-norm, reach a prescribed tolerance, i.e.,
\begin{eqnarray*}
\left\|\mathrm{Res}\left(\varphi_h^{n,i},\mu_h^{n,i}, \bm u_h^{n,i}\right)\right\|_{2}&\leq& \mathrm{Tol}_{\mathrm{res},\mathrm{abs}},\\
\frac{\left\|\mathrm{Res}\left(\varphi_h^{n,i},\mu_h^{n,i}, \bm u_h^{n,i}\right)\right\|_{2}}{\left\|\mathrm{Res}\left(\varphi_h^{n,0},\mu_h^{n,0}, \bm u_h^{n,0}\right)\right\|_{2}}&\leq& \mathrm{Tol}_{\mathrm{res},\mathrm{rel}},\\
\left\|\varphi_h^{n,i}-\varphi_h^{n,i-1}\right\|_{L^2(\Omega)} + \left\|\mu_h^{n,i}-\mu_h^{n,i-1}\right\|_{L^2(\Omega)}+\left\|\bm u_h^{n,i}-\bm u_h^{n,i-1}\right\|_{L^2(\Omega)}&\leq& \mathrm{Tol}_{\mathrm{inc},\mathrm{abs}},\\
\frac{\left\|\varphi_h^{n,i}-\varphi_h^{n,i-1}\right\|_{L^2(\Omega)}}{\left\|\varphi_h^{n,1}-\varphi_h^{n,0}\right\|_{L^2(\Omega)}} + \frac{\left\|\mu_h^{n,i}-\mu_h^{n,i-1}\right\|_{L^2(\Omega)}}{\left\|\mu_h^{n,1}-\mu_h^{n,0}\right\|_{L^2(\Omega)}}+ \frac{\left\|\bm u_h^{n,i}-\bm u_h^{n,i-1}\right\|_{L^2(\Omega)}}{\left\|\bm u_h^{n,1}-\bm u_h^{n,0}\right\|_{L^2(\Omega)}}&\leq& \mathrm{Tol}_{\mathrm{inc},\mathrm{rel}},
\end{eqnarray*}
where $\mathrm{Res}\left(\varphi_h^{n,i},\mu_h^{n,i}, \bm u_h^{n,i}\right)$ is the algebraic residual corresponding to the discretized system of equations.
For all test cases that we run in this paper, $\mathrm{Tol}_{\mathrm{res},\mathrm{abs}}$, $\mathrm{Tol}_{\mathrm{res},\mathrm{rel}}$, $\mathrm{Tol}_{\mathrm{inc},\mathrm{abs}}$, and $\mathrm{Tol}_{\mathrm{inc},\mathrm{rel}}$ are set to $1e-6$.  Moreover, the parameter $\theta$ in the modification to the standard double-well potential and the related convex-concave splitting, see Assumption (A1), is chosen as $\theta=2$.
\begin{remark}
The Cahn-Hilliard subproblem is nonlinear even though the alternating minimization method is applied. We use the Newton method and iterate until similar tolerances as for the full problem are reached ($1e-6$). One could, however, consider to only perform a single iteration of the Newton method in each alternating minimization iteration instead of iterating until the prescribed tolerance is reached, as done in \cite{illiano2021iterative}, in order to speed up the convergence of the total iterative solver.
\end{remark}

\subsection{Test case with phases separated along the middle}
In this test case we initialize the simulation by separating the phases along the middle of the domain, see Figure~\ref{fig:midsplitsg0}. We take $\bm u_h^{0,0}=0$ as initial guess for displacement in the first time step and impose zero Dirichlet boundary conditions for it on the entire boundary. The model parameters can be found in Table~\ref{tab:1}, with 
\begin{equation*}
    \mathbb{C}_{-1} = 
    \begin{pmatrix}
        100 & 20  & 0\\
        20  & 100 & 0\\
        0   & 0   & 200
    \end{pmatrix}, \quad \mathrm{and}\quad 
    \mathbb{C}_{1} = 
    \begin{pmatrix}
        1   & 0.1 & 0\\
        0.1 & 1   & 0\\
        0   & 0   & 2
    \end{pmatrix},
\end{equation*}
where the elasticity tensors are given in Voigt notation. First, we test with different values for the interfacial tension $\gamma = 1, 5, 10, 50, 100$, and then for different values of the swelling parameter $\xi = 0.1, 0.5, 1, 1.5, 2$.  Simulation results for different values of $\gamma$ are plotted in Figure~\ref{fig:midsplitsg0}--\ref{fig:midsplitsg1000} ($\gamma = 5$), Figure~\ref{fig:midsplitmg0}--\ref{fig:midsplitmg1000} ($\gamma = 10$), and Figure~\ref{fig:midsplitlg0}--\ref{fig:midsplitlg1000} ($\gamma = 100)$. Moreover, in Figure~\ref{fig:energy-midsplit} we see that the energy decays over time, using both the semi-implicit time discretization \eqref{eq:weakch1}--\eqref{eq:weakelastic}, and the implicit one \eqref{eq:weakimpch1}--\eqref{eq:weakimpelastic}, for different time-step sizes, $\gamma = 5$ and $\xi = 1$.

\begin{table}[ht]
\centering
\begin{tabular}{c|c|c|c}
Parameter name & Symbol & Value & Unit\\
\hline
Chemical mobility & $m$ & 1 &$\left[\frac{L^4}{FT}\right]$ \\
Interfacial tension & $\gamma$ & -- & $\left[F\right]$ \\
Time step size & $\tau$ & 1e-5 & $\left[T\right]$\\
Final time & $T$ & 0.01 & $\left[T\right]$\\
Swelling parameter & $\xi$ & --
&[--] \\
Mesh diameter & $h$ & $\frac{\sqrt{2}}{65}$& $\left[L\right]$ \\
Regularization parameter & $\ell$ & $0.02$ & [--] \\
Elasticity tensors & $\mathbb{C}_{-1}, \mathbb{C}_{1}$ & - & $\left[\frac{F}{L^2}\right]$
\end{tabular}
\caption{Table of simulation parameters. Here, $L$ denotes the unit of length, $F$ force, and $T$ time.}
\label{tab:1}
\end{table}

\begin{figure}
\centering
\begin{minipage}{0.5\textwidth}
    \begin{subfigure}{0.24\textwidth}
        \includegraphics[width = \textwidth]{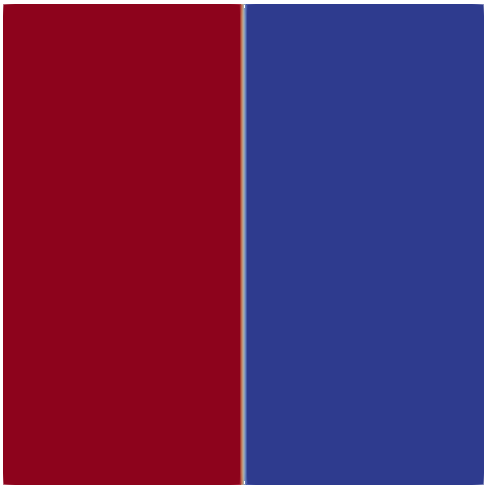}
        \caption{$t = 0$}
        \label{fig:midsplitsg0}
    \end{subfigure}
    \hfill
    \begin{subfigure}{0.24\textwidth}
        \includegraphics[width = \textwidth]{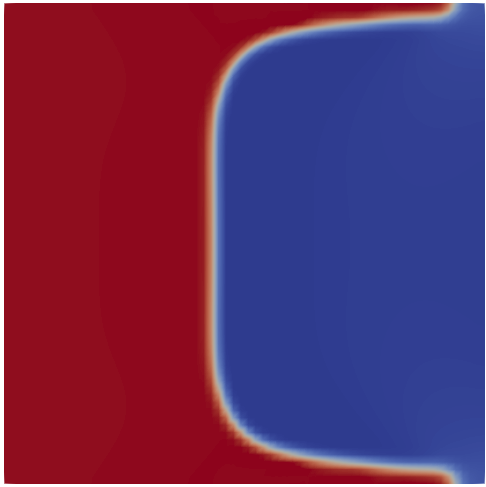}
        \caption{$t = 0.001$}
        \label{fig:midsplitsg100}
    \end{subfigure}
    \hfill
    \begin{subfigure}{0.24\textwidth}
        \includegraphics[width = \textwidth]{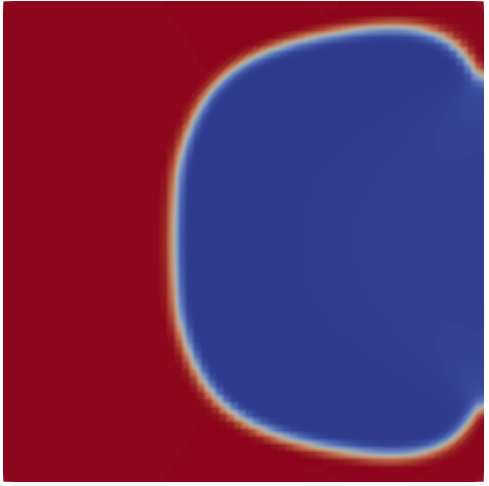}
        \caption{$t = 0.005$}
        \label{fig:midsplitsg500}
    \end{subfigure}
    \hfill
    \begin{subfigure}{0.24\textwidth}
        \includegraphics[width = \textwidth]{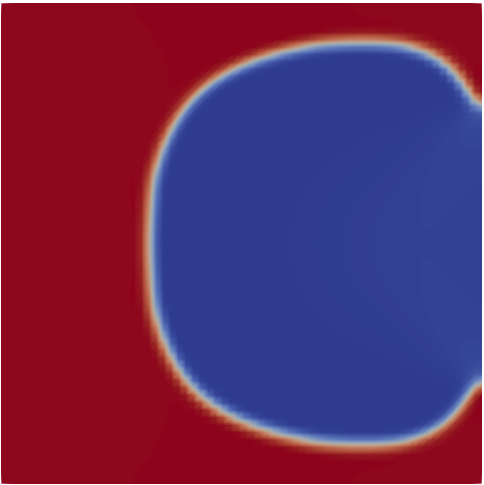}
        \caption{$t = 0.01$}
        \label{fig:midsplitsg1000}
    \end{subfigure}
    \hfill
    \begin{subfigure}{0.24\textwidth}
        \includegraphics[width = \textwidth]{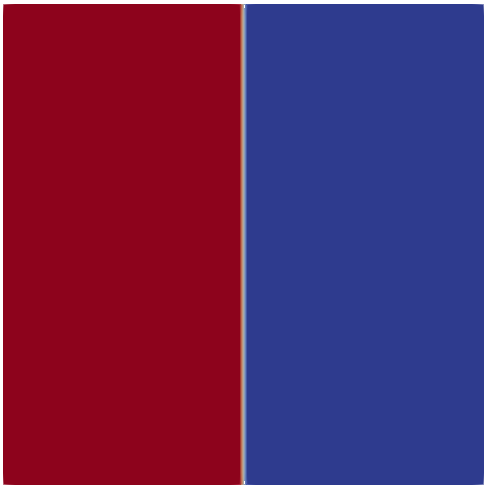}
        \caption{$t = 0$}
        \label{fig:midsplitmg0}
    \end{subfigure}
    \hfill
    \begin{subfigure}{0.24\textwidth}
        \includegraphics[width = \textwidth]{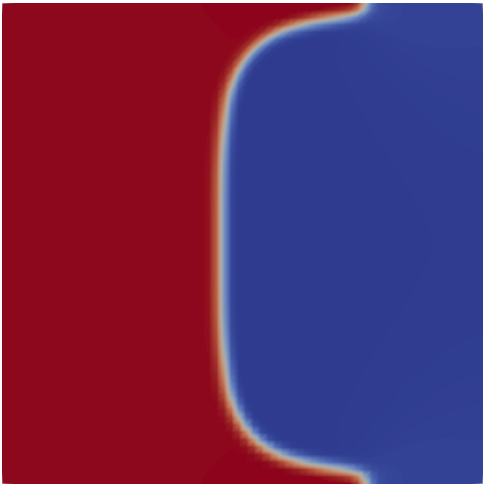}
        \caption{$t = 0.001$}
        \label{fig:midsplitmg100}
    \end{subfigure}
    \hfill
    \begin{subfigure}{0.24\textwidth}
        \includegraphics[width = \textwidth]{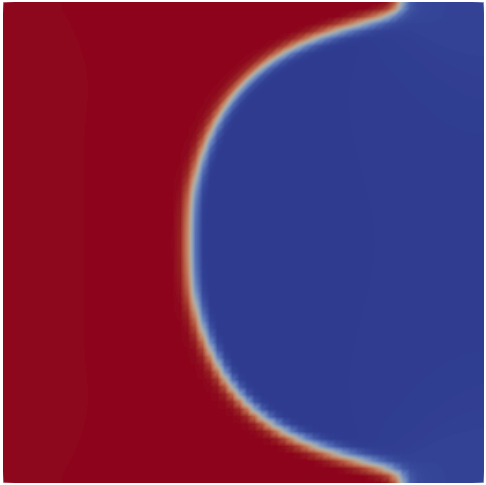}
        \caption{$t = 0.005$}
        \label{fig:midsplitmg500}
    \end{subfigure}
    \hfill
    \begin{subfigure}{0.24\textwidth}
        \includegraphics[width = \textwidth]{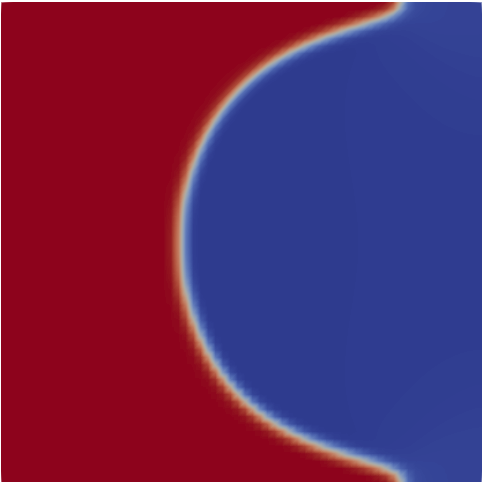}
        \caption{$t = 0.01$}
        \label{fig:midsplitmg1000}
    \end{subfigure}
    \begin{subfigure}{0.24\textwidth}
        \includegraphics[width = \textwidth]{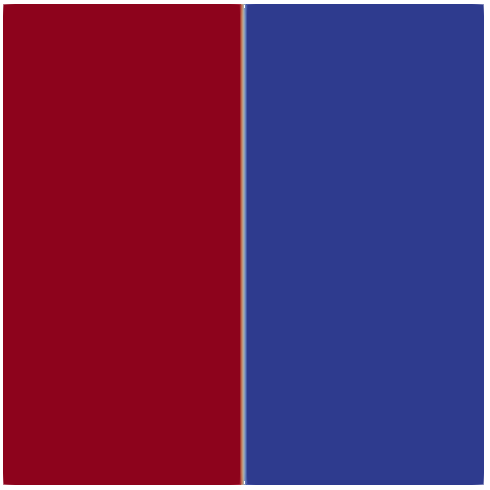}
        \caption{$t = 0$}
        \label{fig:midsplitlg0}
    \end{subfigure}
    \hfill
    \begin{subfigure}{0.24\textwidth}
        \includegraphics[width = \textwidth]{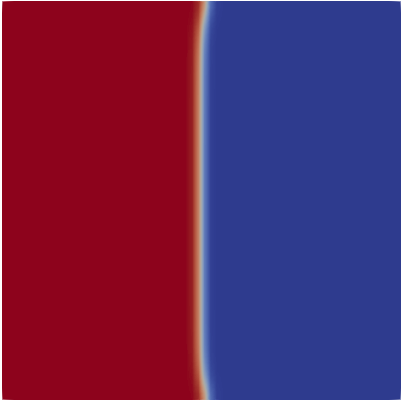}
        \caption{$t = 0.001$}
        \label{fig:midsplitlg100}
    \end{subfigure}
    \hfill
    \begin{subfigure}{0.24\textwidth}
        \includegraphics[width = \textwidth]{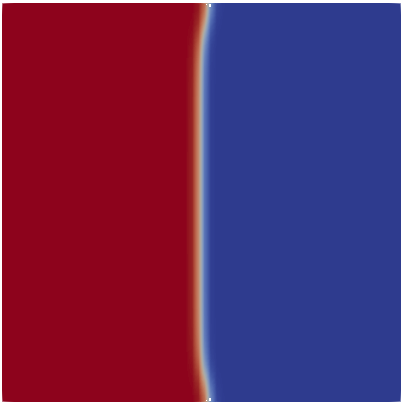}
        \caption{$t = 0.005$}
        \label{fig:midsplitlg500}
    \end{subfigure}
    \hfill
    \begin{subfigure}{0.24\textwidth}
        \includegraphics[width = \textwidth]{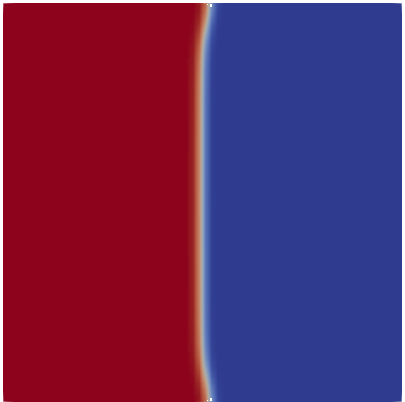}
        \caption{$t = 0.01$}
        \label{fig:midsplitlg1000}
    \end{subfigure}
    \hspace{0cm}\includegraphics[width=\textwidth]{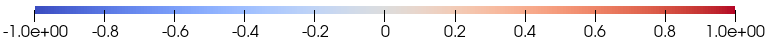}
    \end{minipage}
    \begin{minipage}{0.45\textwidth}
    \begin{subfigure}{\textwidth}
        \input{energy-midsplit.tex}
        \caption{Energy decay over time for different time-step sizes, $\gamma = 5$ and $\xi = 1$.}
        \label{fig:energy-midsplit}
    \end{subfigure}
    \end{minipage}
       \caption{(a) -- (l): the solution at time $t$ for the phase-field $\pf$. (a) -- (d):  $\gamma = 5$, (e) -- (h): $\gamma = 10$, (i) -- (l): $\gamma = 100$. (m): Total energy \eqref{eq:freeenergy} for both the implicit (in the elastic energy) time discretization \eqref{eq:weakimpch1}--\eqref{eq:weakimpelastic} and the semi-implicit one \eqref{eq:weakch1}--\eqref{eq:weakelastic} with different time step sizes and $\gamma = 10$.}
\end{figure}

\subsubsection{Dependence on interfacial tension}\label{sec:midsplitgamma}
We run several simulations with different values for the interfacial tension $\gamma= 1, 5, 10, 50, 100$, while counting the number of iterations the different solution strategies take to achieve satisfactory precision. The other parameters are found in Table~\ref{tab:1}, and the swelling parameter is chosen to be $\xi = 1$.

\pgfplotstableread{
0 0
1 0
2 0
}\datatableEmpty

\pgfplotstableread{
1 0
5 3.354
10 3.265
50 3.036
100 2.481
}\datatableGammaMono

\pgfplotstableread{
1 4.782
5 3.186
10 3.125
50 3.027
100 2.483
}\datatableGammaMonoSemi

\pgfplotstableread{
1 19.787
5 9.412
10 9.552
50 6.966
100 5.280
}\datatableGammaSplit

\pgfplotstableread{
1 11.714
5 6.222
10 5.585
50 5.124
100 4.002
}\datatableGammaSplitSemi

\begin{figure}
  \begin{tikzpicture}
   \pgfplotsset{ybar stacked, ymin=2.500, ymax=20.000, width=\textwidth, enlargelimits=0.1,
   grid=both, grid style={line width=.2pt, draw=gray!20}, 
   xmin=0, xmax = 4
 }
 
 \begin{axis}
 [
 bar width=5pt,
 bar shift=-17.5pt, 
 ylabel={Avg. \# iter. pr. time step}, 
 xticklabels={1, 5, 10, 50, 100},
 xlabel = {Interfacial tension - $\gamma$},
 ytick={3.000, 5.000, 7.000, 9.000, 20.000}, 
 height=5cm,
 xtick=data,
 bar width=5pt,
 bar shift=-7.5pt, 
 height=5cm,
 scaled y ticks = false
 ]  
 
 \pgfplotsinvokeforeach {1}{
     \addplot[color=neworange1, fill=neworange1] table [x expr=\coordindex, y index=#1] {\datatableGammaMono};}
     
\end{axis} 

 \begin{axis}
 [
 bar width=5pt,
 bar shift=-2.5pt, 
 xtick=\empty,
 ytick=\empty, 
 height=5cm
 ]  
 
 \pgfplotsinvokeforeach {1}{
     \addplot[color=newblue1, fill=newblue1] table [x expr=\coordindex, y index=#1] {\datatableGammaMonoSemi};}
     
\end{axis} 

 \begin{axis}
 [
 bar width=5pt,
 bar shift=2.5pt, 
 xtick=\empty,
 ytick=\empty, 
 height=5cm
 ]  
 
 \pgfplotsinvokeforeach {1}{
     \addplot[color=newgreen2, fill=newgreen2] table [x expr=\coordindex, y index=#1] {\datatableGammaSplit};}
     
\end{axis}

\begin{axis}
 [
 bar width=5pt,
 bar shift=7.5pt, 
 xtick=\empty,
 ytick=\empty, 
 height=5cm
 ]  
 
 \pgfplotsinvokeforeach {1}{
     \addplot[color=newgreen1, fill=newgreen1] table [x expr=\coordindex, y index=#1] {\datatableGammaSplitSemi};}
     
\end{axis} 

 \begin{axis}
 [
 bar width=0pt,
 bar shift=10pt, 
 legend columns=2,
 legend style={at={(0.75,1.05)}, anchor=north}, 
 xtick=\empty,
 legend cell align=left,
 legend style={
     /tikz/column 2/.style={
         column sep=-1.3pt,
         row sep=-20pt 
     },
     /tikz/column 6/.style={ 
         column sep=-1.3pt,
         row sep=-20pt 
     }},
 ytick=\empty,
 height=5cm
 ]  
 \pgfplotsinvokeforeach {1}{
     \addplot[color=orange, fill=orange] table [x expr=\coordindex, y index=#1] {\datatableEmpty};}
 \pgfplotsinvokeforeach {1}{
     \addplot[color=newblue1, fill=newblue1] table [x expr=\coordindex, y index=#1] {\datatableEmpty};}
\pgfplotsinvokeforeach {1}{
     \addplot[color=newgreen2, fill=newgreen2] table [x expr=\coordindex, y index=#1] {\datatableEmpty};}
\pgfplotsinvokeforeach {1}{
     \addplot[color=newgreen1, fill=newgreen1] table [x expr=\coordindex, y index=#1] {\datatableEmpty};}

 \legend{
 Imp. Mono, 
 Semi-Imp. Mono., 
 Imp. Split., 
 Semi-Imp. Split.
 }
 \end{axis}

 \end{tikzpicture}
 
 \caption{Test case with phases separated along the middle: Total number of iterations for different values of the interfacial tension parameter $\gamma$. Here, "Imp." refers to the discrete system of equation \eqref{eq:weakimpch1}--\eqref{eq:weakimpelastic}, whereas "Semi.-Imp." corresponds to the discrete system of equations \eqref{eq:weakch1}--\eqref{eq:weakelastic}. Moreover, "Mono." refers to the monolithic full Newton method applied to the discrete system of equations and the alternating minimization algorithm is labeled with "Split.". The numerical scheme \eqref{eq:splitweakch1}--\eqref{eq:splitweakelasticity} corresponds to "Semi-Imp. Split.". Notice that "Imp. Mono." failed to converge for $\gamma = 1$ and, therefore, it is not marked above that value in the plot.}
 \label{fig:gammaiterations}
\end{figure}
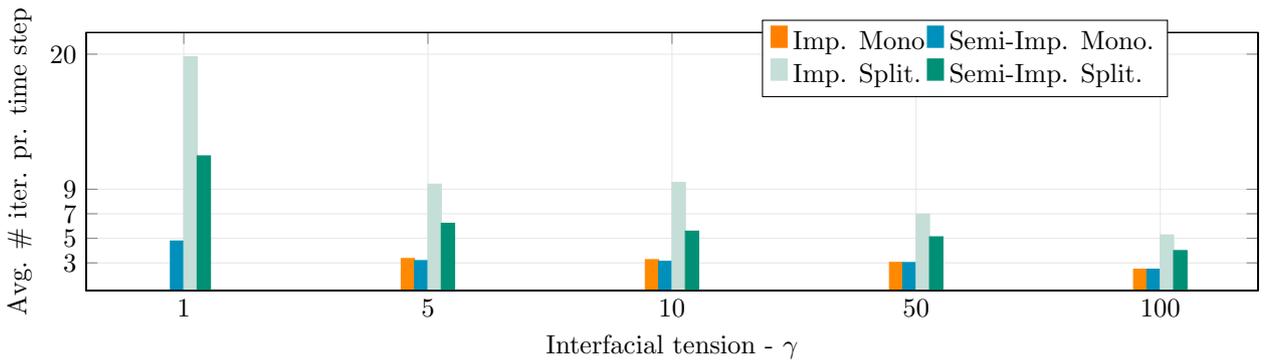

In Figure~\ref{fig:gammaiterations}, we see that the monolithic Newton method converges in fewer iterations than the alternating minimization algorithms. However, for the smallest value of interfacial tension, $\gamma=1$, (when the coupling strength is highest) the monolithic Newton method with implicit-in-time evaluation of the elastic energy \eqref{eq:weakimpch1}--\eqref{eq:weakimpelastic} does not converge at all, and is therefore not a robust choice as a solution strategy. The monolithic Newton method applied to the semi-implicitly discretized system of equations \eqref{eq:weakch1}--\eqref{eq:weakelastic} seems to be a robust choice of linearization procedure, which is due to the convex nature of the related minimization problem, see Proposition~\ref{prop:minimization}. Moreover, as expected from Corollary~\ref{cor:am}, the number of iterations for the alternating minimization method \eqref{eq:splitweakch1}--\eqref{eq:splitweakelasticity} decreases with increasing interfacial tension. This is in fact true for all of the solution strategies as the relative coupling strength between Cahn-Hilliard and elasticity is decreasing for increasing interfacial tension.

\subsubsection{Dependence on swelling parameter}
A similar test is considered for several values of the swelling parameter, $\xi = 0.01, 0.1, 0.5, 1, 1.5, 2$ and a fixed interfacial tension $\gamma = 5$, see Figure~\ref{fig:xiiterations}. Here, we observe, as is expected from the theory, Corollary~\ref{cor:am}, that the coupled problems become more difficult to solve (require more iterations of either the Newton method or alternating minimization) when the swelling parameter increases. This is natural as the swelling parameter is directly connected to the coupling strength between the phase-field and elasticity equations. Another important observation is that for large values of the swelling parameter ($\xi = 1.5$ and $\xi = 2$) the monolithic Newton method applied to the discrete system of equations \eqref{eq:weakimpch1}--\eqref{eq:weakimpelastic} does not converge at all. On the other hand, alternating minimization converges for these cases as well, which (although we have no theoretical proof for it) suggests that the alternating minimization method is more robust than the Newton method for this problem. Notice also that for the smallest value of swelling parameter $\xi = 0.01$, the problem is almost decoupled, and convergence of the linearization/decoupling methods is reached in approximately one iteration (in some time-steps two iterations are required).
\pgfplotstableread{
0 0
1 0
2 0
}\datatableEmpty

\pgfplotstableread{
0.01 1.227
0.1 1.411
0.5 3.233
1 3.354
1.5 0
2 0
}\datatableXiMono

\pgfplotstableread{
0.01 1.227
0.1 1.441
0.5 3.061
1 3.186
1.5 4.141
2 4.410
}\datatableXiMonoSemi

\pgfplotstableread{
0.01 1.268
0.1 1.850
0.5 9.193
1 9.412
1.5 14.884
2 16.902
}\datatableXiSplit

\pgfplotstableread{
0.01 1.265
0.1 1.630
0.5 5.001
1 6.222
1.5 8.717
2 11.279
}\datatableXiSplitSemi

\begin{figure}
  \begin{tikzpicture}
   \pgfplotsset{ybar stacked, ymin=1.000, ymax=17.000, width=\textwidth, enlargelimits=0.05,
   grid=both, grid style={line width=.2pt, draw=gray!20}, 
   xmin=0, xmax = 5
 }
 
 \begin{axis}
 [
 bar width=5pt,
 bar shift=-17.5pt, 
 ylabel={Avg. \# iter. pr. time step}, 
 xticklabels={0.01, 0.1, 0.5, 1, 1.5, 2},
 xlabel = {Swelling parameter - $\xi$},
 ytick={1.250, 3.000, 5.000, 9.000, 16.000}, 
 height=5cm,
 xtick=data,
 bar width=5pt,
 bar shift=-7.5pt, 
 height=5cm,
 scaled y ticks = false
 ]  
 
 \pgfplotsinvokeforeach {1}{
     \addplot[color=neworange1, fill=neworange1] table [x expr=\coordindex, y index=#1] {\datatableXiMono};}
     
\end{axis} 

 \begin{axis}
 [
 bar width=5pt,
 bar shift=-2.5pt, 
 xtick=\empty,
 ytick=\empty, 
 height=5cm
 ]  
 
 \pgfplotsinvokeforeach {1}{
     \addplot[color=newblue1, fill=newblue1] table [x expr=\coordindex, y index=#1] {\datatableXiMonoSemi};}
     
\end{axis} 

 \begin{axis}
 [
 bar width=5pt,
 bar shift=2.5pt, 
 xtick=\empty,
 ytick=\empty, 
 height=5cm
 ]  
 
 \pgfplotsinvokeforeach {1}{
     \addplot[color=newgreen2, fill=newgreen2] table [x expr=\coordindex, y index=#1] {\datatableXiSplit};}
     
\end{axis}

\begin{axis}
 [
 bar width=5pt,
 bar shift=7.5pt, 
 xtick=\empty,
 ytick=\empty, 
 height=5cm
 ]  
 
 \pgfplotsinvokeforeach {1}{
     \addplot[color=newgreen1, fill=newgreen1] table [x expr=\coordindex, y index=#1] {\datatableXiSplitSemi};}
     
\end{axis} 

 \begin{axis}
 [
 bar width=0pt,
 bar shift=10pt, 
 legend columns=2,
 legend style={at={(0.25,1.05)}, anchor=north}, 
 xtick=\empty,
 legend cell align=left,
 legend style={
     /tikz/column 2/.style={
         column sep=-1.3pt,
         row sep=-20pt 
     },
     /tikz/column 6/.style={ 
         column sep=-1.3pt,
         row sep=-20pt 
     }},
 ytick=\empty,
 height=5cm
 ]  
 \pgfplotsinvokeforeach {1}{
     \addplot[color=orange, fill=orange] table [x expr=\coordindex, y index=#1] {\datatableEmpty};}
 \pgfplotsinvokeforeach {1}{
     \addplot[color=newblue1, fill=newblue1] table [x expr=\coordindex, y index=#1] {\datatableEmpty};}
\pgfplotsinvokeforeach {1}{
     \addplot[color=newgreen2, fill=newgreen2] table [x expr=\coordindex, y index=#1] {\datatableEmpty};}
\pgfplotsinvokeforeach {1}{
     \addplot[color=newgreen1, fill=newgreen1] table [x expr=\coordindex, y index=#1] {\datatableEmpty};}

 \legend{
 Imp. Mono, 
 Semi-Imp. Mono., 
 Imp. Split., 
 Semi-Imp. Split.
 }
 \end{axis}

 \end{tikzpicture}
 
 \caption{Test case with phases separated along the middle: Total number of iterations for different values of the swelling parameter $\xi$. Here, "Imp." refers to the discrete system of equation \eqref{eq:weakimpch1}--\eqref{eq:weakimpelastic}, whereas "Semi.-Imp." corresponds to the discrete system of equations \eqref{eq:weakch1}--\eqref{eq:weakelastic}.  Moreover, "Mono." refers to the monolithic full Newton method applied to the discrete system of equations and the alternating minimization algorithm is labeled with "Split.". The numerical scheme \eqref{eq:splitweakch1}--\eqref{eq:splitweakelasticity} corresponds to "Semi-Imp. Split.". Notice that "Imp. Mono." failed to converge for $\xi = 1.5$ and $\xi = 2$ and, therefore, it is not marked above those values in the plot.}
 \label{fig:xiiterations}
\end{figure}
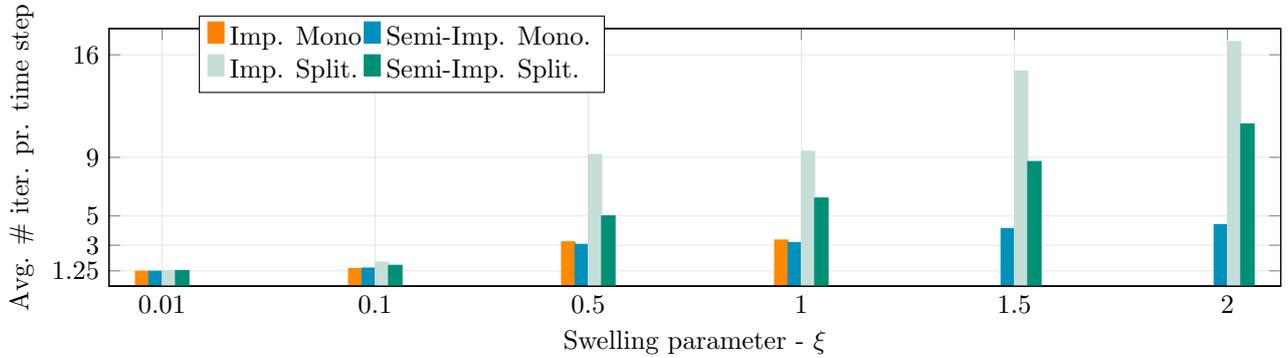

\subsubsection{Anderson acceleration applied to the decoupling algorithms}
As mentioned in the introduction, the Anderson acceleration \cite{anderson1965iterative} has been successfully applied to accelerate decoupling/splitting schemes, as alternating minimization previously \cite{storvik2021accelerated, both2019anderson}, or linearly convergence methods like the Picard algorithm for Navier-Stokes \cite{pollock2019anderson}. The scheme is applied as a post-process to fixed-point iterations and updates the current iterate as a linear combination of the $m$ (called depth of the acceleration) previous iterates. More careful explanation of the method can be found in e.g., \cite{storvik2021accelerated, both2019anderson, pollock2019anderson}.

Here we applied the Anderson acceleration to accelerate the alternating minimization method \eqref{eq:splitweakch1}--\eqref{eq:splitweakelasticity} ("Semi-Imp. Split."), and the staggered scheme applied to \eqref{eq:weakimpch1}--\eqref{eq:weakimpelastic} ("Imp. Split."). Simulation parameters from Table~\ref{tab:1} with $\gamma = 1$ and $\xi = 1$ are used, similar to the first column in Figure~\ref{fig:gammaiterations}, and we test for acceleration depths ranging from $m= 0$ (no acceleration) to $m= 5$. The results are displayed in Figure~\ref{fig:midsplitAAiterations}. We observe that for the staggered scheme applied to \eqref{eq:weakimpch1}--\eqref{eq:weakimpelastic} ("Imp. Split."), the postprocessing accelerates the convergence quite significantly, however, it fails to converge for the largest depth ($m=5$). For the  the alternating minimization method \eqref{eq:splitweakch1}--\eqref{eq:splitweakelasticity} ("Semi-Imp. Split."), it only accelerates slightly, and actually decelerates the convergence for larger values of depths ($m= 4,5$). Therefore, using the Anderson acceleration to solve the alternating minimization problem might be beneficial for smaller depths. Moreover, there are several ways of improving the convergence of the Anderson acceleration, e.g, periodically restart it from depth $m=0$ or turn it on and off using some safeguard mechanics (see \cite{storvik2021accelerated}), but this is outside the scope of the current paper to investigate.
\pgfplotstableread{
0 0
1 0
2 0
}\datatableEmpty

\pgfplotstableread{
0 19.787
1 13.703
2 13.108
3 13.324
4 13.730
5 0
}\datatableAAMidsplitImp

\pgfplotstableread{
0 11.714
1 10.018
2 9.651
3 11.115
4 12.231
5 13.477
}\datatableAAMidsplitSemiImp

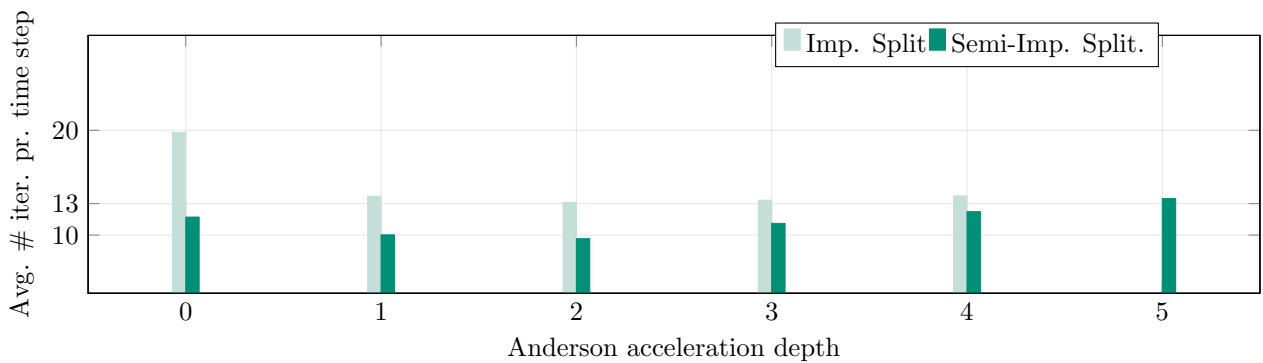
\begin{figure}
  \begin{tikzpicture}
   \pgfplotsset{ybar stacked, ymin=6.500, ymax=27.000, width=\textwidth, enlargelimits=0.1,
   grid=both, grid style={line width=.2pt, draw=gray!20}, 
   xmin=0, xmax = 5
 }
 
 \begin{axis}
 [
 bar width=5pt,
 bar shift=-2.5pt, 
 ylabel={Avg. \# iter. pr. time step}, 
 xticklabels={0, 1, 2, 3, 4, 5},
 xlabel = {Anderson acceleration depth},
 ytick={10.000, 13.000, 20.000}, 
 height=5cm,
 xtick=data,
 height=5cm,
 scaled y ticks = false
 ]  
 
 \pgfplotsinvokeforeach {1}{
     \addplot[color=newgreen2, fill=newgreen2] table [x expr=\coordindex, y index=#1] {\datatableAAMidsplitImp};}
     
\end{axis} 

 \begin{axis}
 [
 bar width=5pt,
 bar shift=2.5pt, 
 xtick=\empty,
 ytick=\empty, 
 height=5cm
 ]  
 
 \pgfplotsinvokeforeach {1}{
     \addplot[color=newgreen1, fill=newgreen1] table [x expr=\coordindex, y index=#1] {\datatableAAMidsplitSemiImp};}
     
\end{axis} 

 \begin{axis}
 [
 bar width=0pt,
 bar shift=10pt, 
 legend columns=2,
 legend style={at={(0.75,1.05)}, anchor=north}, 
 xtick=\empty,
 legend cell align=left,
 legend style={
     /tikz/column 2/.style={
         column sep=-1.3pt,
         row sep=-20pt 
     },
     /tikz/column 6/.style={ 
         column sep=-1.3pt,
         row sep=-20pt 
     }},
 ytick=\empty,
 height=5cm
 ]  
 \pgfplotsinvokeforeach {1}{
     \addplot[color=newgreen2, fill=newgreen2] table [x expr=\coordindex, y index=#1] {\datatableEmpty};}
 \pgfplotsinvokeforeach {1}{
     \addplot[color=newgreen1, fill=newgreen1] table [x expr=\coordindex, y index=#1] {\datatableEmpty};}
 \legend{
 Imp. Split, 
 Semi-Imp. Split.
 }
 \end{axis}

 \end{tikzpicture}
 
 \caption{Test case with phases segregated in the middle: Total number of iterations for different Anderson acceleration depths. Here, "Imp." refers to the discrete system of equation \eqref{eq:weakimpch1}--\eqref{eq:weakimpelastic}, whereas "Semi.-Imp." corresponds to the discrete system of equations \eqref{eq:weakch1}--\eqref{eq:weakelastic}. Notice that "Semi-Imp. Split." failed to converge for depth $5$ and, therefore, it is not marked above that value in the plot.}
 \label{fig:midsplitAAiterations}
\end{figure}

\subsection{Random initial conditions: Spinodal decomposition}
We provide another numerical experiment here, with randomized initial conditions, where the initial "mixture" decomposes into pure phases and we observe a coarsening effect that resembles spinodal decomposition. This effect has been studied for the Cahn-Larch\'e equations previously in e.g., \cite{graser2014numerical, garcke2001cahn}. In Figure~\ref{fig:randomfig} we present simulation results using parameters from Table~\ref{tab:1}, $\xi = 1$ and $\gamma = 5,$ $\gamma = 10,$ and $\gamma = 100$. In Figure~\ref{fig:energy-random}, we plot the total energy \eqref{eq:freeenergy} of the system for both the discrete system of equations \eqref{eq:weakimpch1}--\eqref{eq:weakimpelastic} ("Imp.\!") and \eqref{eq:weakch1}--\eqref{eq:weakelastic} ("Semi-Imp.\!") for different values of the interfacial tension parameter. We observe that there is close to no difference between the free energy over the simulation for the two time-discretizations and that both of them are decreasing over time.

In Figure~\ref{fig:randomgammaiterations}, the total number of iterations for the different solution strategies are presented for different values of the interfacial tension $\gamma = 1, 5, 10, 50, 100$. We see that, as in Section~\ref{sec:midsplitgamma}, the number of decoupling/linearization iterations decrease for increasing values of the interfacial tension, exactly as the theory for alternating minimization predicts, Corollary~\ref{cor:am}. Again the Newton method outperforms the alternating minimization method in terms of numbers of iterations, although the difference shrinks significantly for lower relative coupling strengths ($\gamma$ increasing). Moreover, we stress that the alternating minimization method has the added benefit of allowing for the use of readily available implementations and solvers for Cahn-Hilliard and elasticity with only small modifications.

\begin{figure}
\centering
\begin{minipage}{0.5\textwidth}
    \begin{subfigure}{0.24\textwidth}
        \includegraphics[width = \textwidth]{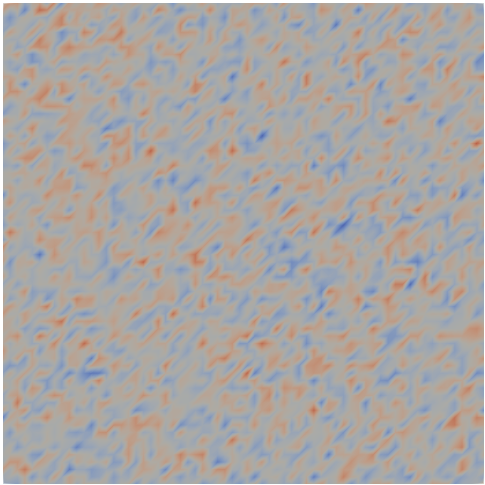}
        \caption{$t = 0$}
        \label{fig:randomg0}
    \end{subfigure}
    \hfill
    \begin{subfigure}{0.24\textwidth}
        \includegraphics[width = \textwidth]{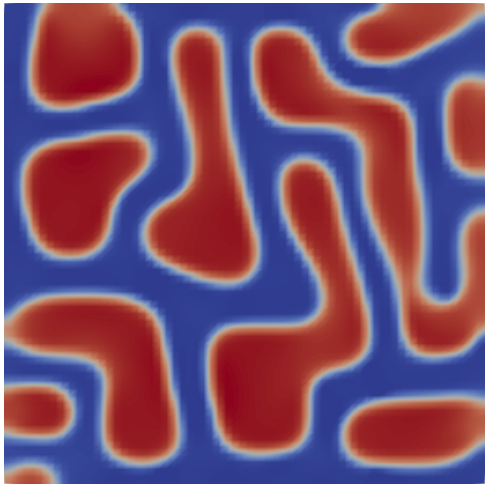}
        \caption{$t = 0.0004$}
        \label{fig:randomg4}
    \end{subfigure}
    \hfill
    \begin{subfigure}{0.24\textwidth}
        \includegraphics[width = \textwidth]{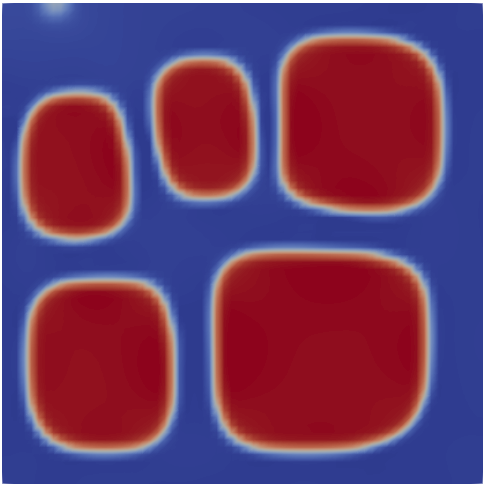}
        \caption{$t = 0.01$}
        \label{fig:randomg50}
    \end{subfigure}
    \hfill
    \begin{subfigure}{0.24\textwidth}
        \includegraphics[width = \textwidth]{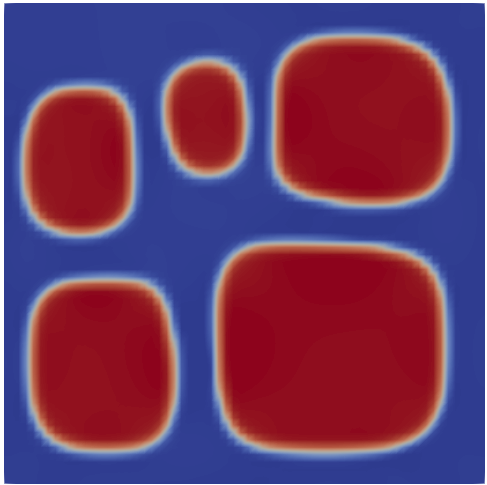}
        \caption{$t = 0.02$}
        \label{fig:randomg1000}
    \end{subfigure}
    \hfill
    \begin{subfigure}{0.24\textwidth}
        \includegraphics[width = \textwidth]{figs/random-0-gammaone.png}
        \caption{$t = 0$}
        \label{fig:randomteng0}
    \end{subfigure}
    \hfill
    \begin{subfigure}{0.24\textwidth}
        \includegraphics[width = \textwidth]{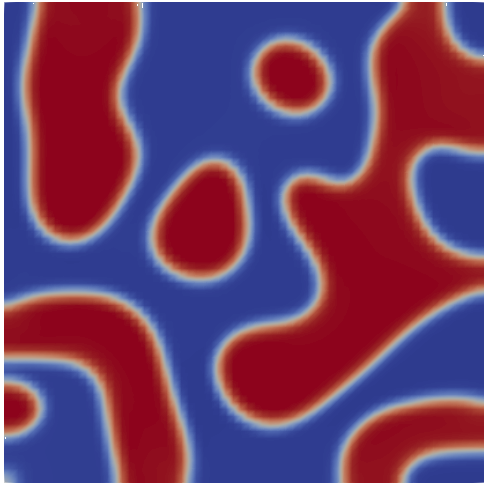}
        \caption{$t = 0.0004$}
        \label{fig:randomgten4}
    \end{subfigure}
    \hfill
    \begin{subfigure}{0.24\textwidth}
        \includegraphics[width = \textwidth]{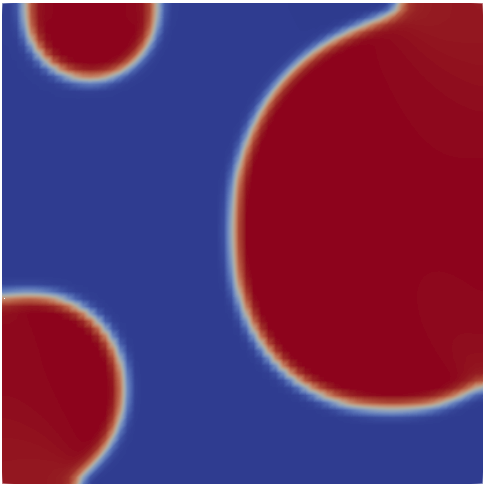}
        \caption{$t = 0.01$}
        \label{fig:randomteng50}
    \end{subfigure}
    \hfill
    \begin{subfigure}{0.24\textwidth}
        \includegraphics[width = \textwidth]{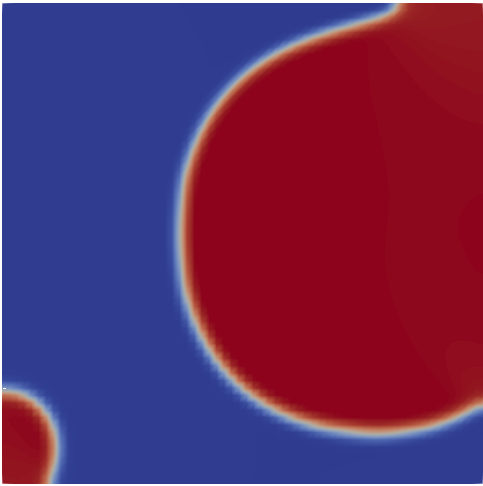}
        \caption{$t = 0.02$}
        \label{fig:randomteng1000}
    \end{subfigure}
    \begin{subfigure}{0.24\textwidth}
        \includegraphics[width = \textwidth]{figs/random-0-gammaone.png}
        \caption{$t = 0$}
        \label{fig:randomhundredg0}
    \end{subfigure}
    \hfill
    \begin{subfigure}{0.24\textwidth}
        \includegraphics[width = \textwidth]{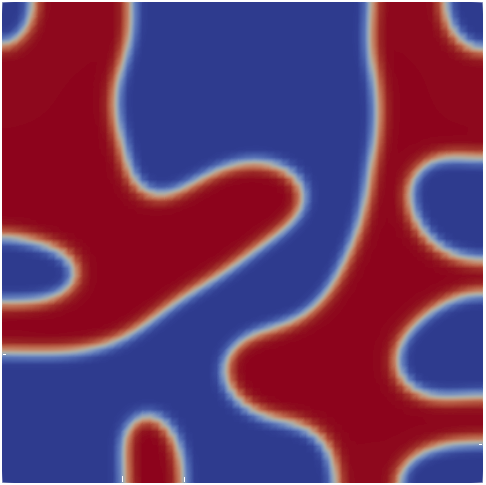}
        \caption{$t = 0.0004$}
        \label{fig:randomhundredg4}
    \end{subfigure}
    \hfill
    \begin{subfigure}{0.24\textwidth}
        \includegraphics[width = \textwidth]{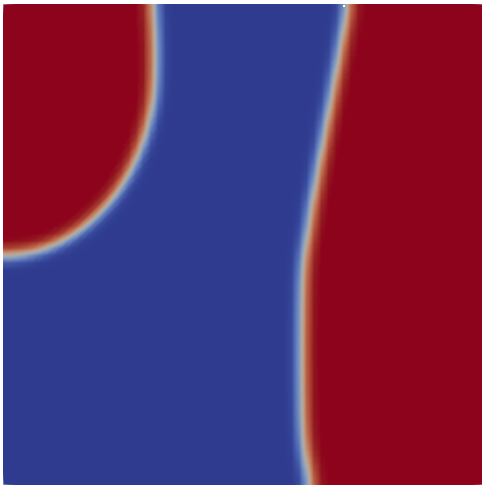}
        \caption{$t = 0.01$}
        \label{fig:randomhundredg50}
    \end{subfigure}
    \hfill
    \begin{subfigure}{0.24\textwidth}
        \includegraphics[width = \textwidth]{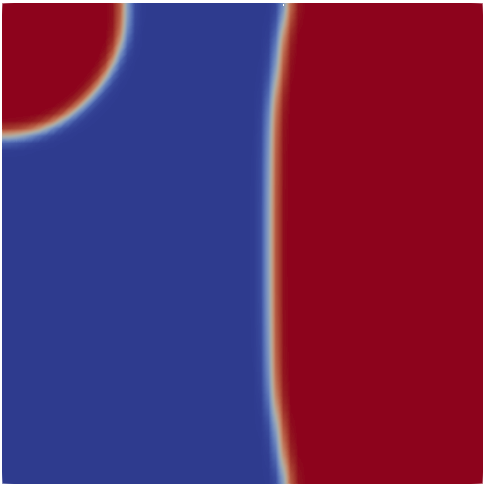}
        \caption{$t = 0.02$}
        \label{fig:randomhundredg1000}
    \end{subfigure}
    \hspace{0cm}\includegraphics[width=\textwidth]{figs/colorbarhorizontal.png}
    \end{minipage}
    \begin{minipage}{0.45\textwidth}
    \begin{subfigure}{\textwidth}
        \input{energy-random.tex}
        \caption{Energy \eqref{eq:freeenergy} decay over time for both the discrete method \eqref{eq:weakimpch1}--\eqref{eq:weakimpelastic} ("Imp.\!"), and \eqref{eq:weakch1}--\eqref{eq:weakelastic} ("Semi-Imp.\!") for different values of interfacial tension parameter. Notice the logarithmic scale of the y-axis.}
        \label{fig:energy-random}
    \end{subfigure}
    \end{minipage}
       \caption{(a) -- (l): the solution at time $t$ for the phase-field $\pf$. (a) -- (d):  $\gamma = 5$, (e) -- (h): $\gamma = 10$, (i) -- (l): $\gamma = 100$. (m): Total energy \eqref{eq:freeenergy} for both the implicit (in the elastic energy) time discretization \eqref{eq:weakimpch1}--\eqref{eq:weakimpelastic} and the semi-implicit one \eqref{eq:weakch1}--\eqref{eq:weakelastic} with different time step sizes and $\gamma = 10$.}
       \label{fig:randomfig}
\end{figure}

\pgfplotstableread{
0 0
1 0
2 0
}\datatableEmpty





\pgfplotstableread{
1 3.688
5 3.896
10 3.893
50 3.586
100 3.298
}\datatableRandomGammaMono

\pgfplotstableread{
1 4.016
5 3.293
10 3.363
50 3.521
100 3.292
}\datatableRandomGammaMonoSemi

\pgfplotstableread{
1 13.447
5 8.487
10 8.431
50 5.746
100 5.004
}\datatableRandomGammaSplit

\pgfplotstableread{
1 11.776
5 6.447
10 5.790
50 4.426
100 4.082
}\datatableRandomGammaSplitSemi

\begin{figure}
  \begin{tikzpicture}
   \pgfplotsset{ybar stacked, ymin=3.250, ymax=13.000, width=\textwidth, enlargelimits=0.1,
   grid=both, grid style={line width=.2pt, draw=gray!20}, 
   xmin=0, xmax = 4
 }
 
 \begin{axis}
 [
 bar width=5pt,
 bar shift=-17.5pt, 
 ylabel={Avg. \# iter. pr. time step}, 
 xticklabels={1, 5, 10, 50, 100},
 xlabel = {Interfacial tension - $\gamma$},
 ytick={3.000, 5.000, 7.000, 9.000, 20.000}, 
 height=5cm,
 xtick=data,
 bar width=5pt,
 bar shift=-7.5pt, 
 height=5cm,
 scaled y ticks = false
 ]  
 
 \pgfplotsinvokeforeach {1}{
     \addplot[color=neworange1, fill=neworange1] table [x expr=\coordindex, y index=#1] {\datatableRandomGammaMono};}
     
\end{axis} 

 \begin{axis}
 [
 bar width=5pt,
 bar shift=-2.5pt, 
 xtick=\empty,
 ytick=\empty, 
 height=5cm
 ]  
 
 \pgfplotsinvokeforeach {1}{
     \addplot[color=newblue1, fill=newblue1] table [x expr=\coordindex, y index=#1] {\datatableRandomGammaMonoSemi};}
     
\end{axis} 

 \begin{axis}
 [
 bar width=5pt,
 bar shift=2.5pt, 
 xtick=\empty,
 ytick=\empty, 
 height=5cm
 ]  
 
 \pgfplotsinvokeforeach {1}{
     \addplot[color=newgreen2, fill=newgreen2] table [x expr=\coordindex, y index=#1] {\datatableRandomGammaSplit};}
     
\end{axis}

\begin{axis}
 [
 bar width=5pt,
 bar shift=7.5pt, 
 xtick=\empty,
 ytick=\empty, 
 height=5cm
 ]  
 
 \pgfplotsinvokeforeach {1}{
     \addplot[color=newgreen1, fill=newgreen1] table [x expr=\coordindex, y index=#1] {\datatableRandomGammaSplitSemi};}
     
\end{axis} 

 \begin{axis}
 [
 bar width=0pt,
 bar shift=10pt, 
 legend columns=2,
 legend style={at={(0.75,1.05)}, anchor=north}, 
 xtick=\empty,
 legend cell align=left,
 legend style={
     /tikz/column 2/.style={
         column sep=-1.3pt,
         row sep=-20pt 
     },
     /tikz/column 6/.style={ 
         column sep=-1.3pt,
         row sep=-20pt 
     }},
 ytick=\empty,
 height=5cm
 ]  
 \pgfplotsinvokeforeach {1}{
     \addplot[color=orange, fill=orange] table [x expr=\coordindex, y index=#1] {\datatableEmpty};}
 \pgfplotsinvokeforeach {1}{
     \addplot[color=newblue1, fill=newblue1] table [x expr=\coordindex, y index=#1] {\datatableEmpty};}
\pgfplotsinvokeforeach {1}{
     \addplot[color=newgreen2, fill=newgreen2] table [x expr=\coordindex, y index=#1] {\datatableEmpty};}
\pgfplotsinvokeforeach {1}{
     \addplot[color=newgreen1, fill=newgreen1] table [x expr=\coordindex, y index=#1] {\datatableEmpty};}

 \legend{
 Imp. Mono, 
 Semi-Imp. Mono., 
 Imp. Split., 
 Semi-Imp. Split.
 }
 \end{axis}

 \end{tikzpicture}
 
 \caption{Test case with random initial data: Total number of iterations for different values of the interfacial tension parameter $\gamma$. Here, "Imp." refers to the discrete system of equation \eqref{eq:weakimpch1}--\eqref{eq:weakimpelastic}, whereas "Semi.-Imp." corresponds to the discrete system of equations \eqref{eq:weakch1}--\eqref{eq:weakelastic}. Moreover, "Mono." refers to the monolithic full Newton method applied to the discrete system of equations and the alternating minimization algorithm is labeled with "Split.". The numerical scheme \eqref{eq:splitweakch1}--\eqref{eq:splitweakelasticity} corresponds to "Semi-Imp. Split.".}
 \label{fig:randomgammaiterations}
\end{figure}
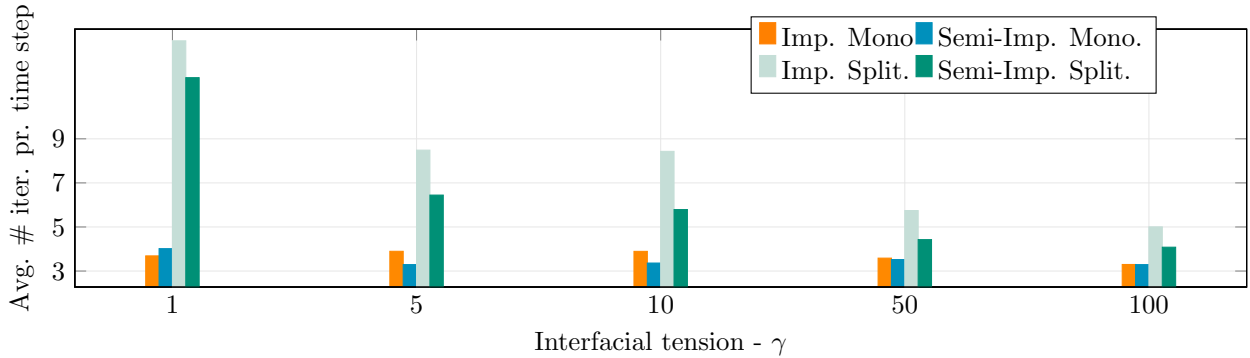

\section{Conclusions}\label{sec:conclusion}
In this paper, we proposed a semi-implicit time discretization to the Cahn-Larch\'e equations and showed that it is equivalent to a convex minimization problem. Then convergence of alternating minimization applied to this problem was proved, and several numerical experiments to study its convergence properties in comparison to the monolithic Newton method were provided. Additionally, the alternating minimization (splitting method) and the monolithic Newton method applied to the newly proposed semi-implicit time-discretization were compared numerically to the same iterative methods applied to a more standard choice of time-discretization with implicit-in-time evaluations of the elastic contributions and a convex-concave split of the double-well potential. We observed that the convergence properties of the iterative methods (Newton's method and alternating minimization) applied to the newly proposed time-discretization are superior to those that are applied to the standard discretization, and in several cases we get convergence of the Newton method for the former and not for the latter. Moreover, for the special case of phase-field-independent elasticity tensor we proved that the discretization is unconditionally gradient stable, by exploiting its minimization structure. For the phase-field dependent elasticity tensor, numerical experiments show that the free energy of the system decreases over time. The newly proposed time-discretization is shown to be well suited for iterative solution schemes and provides a needed alternative to the standard implicit methods.

\section*{Acknowledgments}
The work has been partly supported by the Centre for Sustainable Subsurface Resources, funded by the Norwegian Research council, as well as the FracFlow project funded by Equinor, Norway through Akademiaavtalen.

\bibliographystyle{unsrt}
\bibliography{bibliography.bib}

\end{document}